\documentclass[reqno,11pt]{amsart}
\usepackage{amsmath,amsxtra,amsbsy,enumitem, latexsym, amsfonts, amssymb, amsthm, amscd, stmaryrd}
\setlist[itemize]{leftmargin=0.7cm}
\usepackage{hyperref}
\usepackage{graphics,epsf,psfrag}
\usepackage{bbm, dsfont}
\usepackage{mathtools}
\usepackage{xcolor}
\usepackage{float}
\usepackage{tikz}
\usepackage[headings]{fullpage}
\usepackage{colonequals} 
\usepackage{microtype}
\usepackage{enumitem}
\usepackage{mathrsfs}

\numberwithin{equation}{section}

\usepackage{subcaption}
\definecolor{darkred}{rgb}{0.7,0.1,0.1}

\definecolor{darkgreen}{rgb}{0.1,0.7,0.1}

\newcommand{\bbE}{{\ensuremath{\mathbb E}} }

\newcommand{\bbP}{{\ensuremath{\mathbb P}} }

\renewcommand{\epsilon}{\varepsilon}
\newcommand{\eps}{\varepsilon}

\newcommand{\ga}{\alpha}

\newcommand{\gga}{\gamma}            

\newcommand{\gd}{\delta}
\newcommand{\gth}{\vartheta}

\newcommand{\gz}{\zeta}

\newcommand{\go}{\omega}

\newcommand{\gO}{\Omega}

\newcommand{\gL}{\Lambda}

\newcommand{\dd}{\mathrm{d}}

\newcommand{\cC}{{\ensuremath{\mathcal C}} }

\theoremstyle{plain}
\newtheorem{theorem}{Theorem}[section]
\newtheorem{lemma}[theorem]{Lemma}

\newtheorem{proposition}[theorem]{Proposition}
\newtheorem{corollary}[theorem]{Corollary}

\newtheorem{assumption}[theorem]{Assumption}
\newtheorem{remark}[theorem]{Remark}

\renewcommand{\bar}{\overline}

\DeclareFontFamily{U}{mathx}{}
\DeclareFontShape{U}{mathx}{m}{n}{<-> mathx10}{}
\DeclareSymbolFont{mathx}{U}{mathx}{m}{n}
\DeclareMathAccent{\widehat}{0}{mathx}{"70}
\DeclareMathAccent{\widecheck}{0}{mathx}{"71}

\newcommand{\R}{\mathbb{R}}

\newcommand{\N}{\mathbb{N}}
\newcommand{\ceq}{\coloneqq}
\newcommand{\eqc}{\eqqcolon}

\newcommand{\ind}{{{\mathbbm 1}}}

\newcommand{\ee}{\mathrm{e}}
\newcommand{\probd}{\mathbb{P}}
\newcommand{\Rl}{\mathbb{R}}
\newcommand{\prob}{\mathsf{P}}
\newcommand{\Ex}{\mathsf{E}}
\newcommand{\Exd}{\mathbb{E}}

\newcommand{\pae}{\mathbb{P}\mbox{-a.e.\ }\omega}
\newcommand{\cov}{\mathsf{cov}}
\newcommand{\var}{\mathsf{var}}
\newcommand{\ii}{\mathrm{i}}

\title[The localized phase  of pinning models with correlated Gaussian disorder]{The localized phase  of pinning models \\ with correlated Gaussian disorder}

\author[G. Giacomin]{Giambattista Giacomin}
\address{Universit\`a di Padova,
Dipartimento di Matimatica Tullio Levi-Civita, Via Trieste 63,
35123 Padova, Italy} \email{giacomin@math.unipd.it}

\author[A. Legrand]{Alexandre legrand}
\email{legrand@math.unipd.it}

\author[M. Zamparo]{Marco Zamparo}
\address{Universit\`a degli Studi del Piemonte Orientale, Dipartimento di Scienze e Innovazione Tecnologica, Viale Teresa Michel 11, I-15121 Alessandria, Italy} \email{marco.zamparo@uniupo.it}

\begin{document}

\begin{abstract}
We demonstrate that the results established for the localized regime
of the pinning model with independent disorder -- notably the
$\cC^\infty$ regularity of the free energy, the scaling of the largest
excursion between pinned sites, and the Central Limit Theorems for
both the contact number and its mean -- can be generalized to
translation-ergodic environments under the hypothesis that disorder is
Gaussian. Our results are obtained assuming only summability of the
charges' covariances.  The two key ingredients for our proofs are the
Birkhoff-sum approach introduced in~\cite{GZ25concentration} for
independent disorder, but uniquely suited to handle correlated
environments, and decorrelation tools such as the general and powerful
Nelson's Gaussian hypercontractivity. We also rely on the Bernstein
blocking method to prove a Central Limit Theorem under dependence.
Additionally, we develop further techniques specifically tailored to
the one-dimensional structure of the pinning model.
\medskip

\noindent  \emph{AMS  subject classification (2020 MSC)}:
60K37,  
82B44, 
60K35, 
60F05  	

\end{abstract}
\maketitle

\section{Localization in the pinning model}
\subsection{The pinning model}

Pinning (or wetting) models were introduced to study the effect of a
defect region on a fluctuating string in a variety of contexts
\cite{G07book,dh07,PS70,Vel06}. A concrete example is a
\emph{directed polymer} interacting with a \emph{defect line}, where
the chain may be either attracted or repelled by the line. If the
monomers are not identical, or the defect line is not uniform, then
the interaction strength may vary spatially, rendering the model
non-homogeneous.  The inhomogeneous character of the model is
typically designed such that a homogeneous behavior emerges on a large
scale: this is often achieved by sampling the characteristics of the
monomers or the defect line from a sequence of independent and
identically distributed (i.i.d.) random variables (\emph{the
disorder}). In this way, the model remains homogeneous in a
statistical sense -- and thus on large scales -- while being
potentially highly inhomogeneous at the microscopic level. This
property also holds if i.i.d.\ disorder is replaced by
translation-ergodic disorder; however, correlated disorder generally
complicates the analysis and may lead to fundamentally different model
behaviors \cite{WH83}.

Pinning models may undergo a \emph{localization phase transition}
between a delocalized phase, where the chain has negligible contacts
with the defect line, and a localized regime, where the number of
contacts is proportional to the total length of the chain. This
transition has attracted significant interest in both pure mathematics
and applied sciences, notably physics and biology (see
\cite{DHV92,FLNO86, G07book, Gia11, dh07, KMP00, PS70, Vel06} and
references therein). In this paper we investigate cases involving
correlated disorder, focusing on fine properties within the localized
phase. Our goal is to obtain results under minimal assumptions on the
correlations and, to this aim, we focus on Gaussian disorder.

\smallskip
     
As explained in detail for example in \cite{G07book, Gia11,
  dh07,Vel06}, pinning models may be written in mathematical terms as
Gibbs measures built on renewal sequences and this is the way we
introduce them here.  Let $T_1,T_2,\ldots$ be i.i.d.\ random
variables valued in $\N\ceq\{1,2,\ldots\}$. Set $S_0\ceq0$ and
$S_i\ceq T_1+\cdots+T_i$ for $i\in\N$: hence $S=(S_i)_{i\in\N_0}$ is
a renewal process in $\N_0\ceq\{0\}\cup\N$, which describes the
monomers from the polymer chain that are in contact with the defect
line. In particular, if we define for $n\ge0$
\begin{equation}
L_n \ceq \sup\big\{i\ge0: S_i\leq n\big\}\quad\text{and}\quad M_n\ceq \sup\big\{T_i: i\leq L_n\big\}\,,
\end{equation}
then $L_n$ is the number of contacts between the chain and the defect
line up to the $n$-th monomer and $M_n$ is the largest gap between
contacts (see Figure~\ref{fig:1}). We denote by $\prob$ the law of the
sequence $(T_i)_{i\in\N}$ and we assume that for all $t\in\N$
\begin{equation}
p(t)\ceq \prob[T_1=t]= \frac{\ell(t)}{t^{\ga+1}}\,,
\end{equation}
where $\ga\geq0$ and $\ell:[1,+\infty)\to[0,+\infty)$ is a slowly
    varying function at infinity (see~\cite{BGT89} for the definition
    and properties of slowly varying functions). We also suppose that
    $p(t)>0$ for all $t\in\N$. We refer to~\cite{G07book, Gia11, dh07}
    for the motivations behind those assumptions.

\begin{figure}[h]
\centering
\includegraphics[width=16 cm]{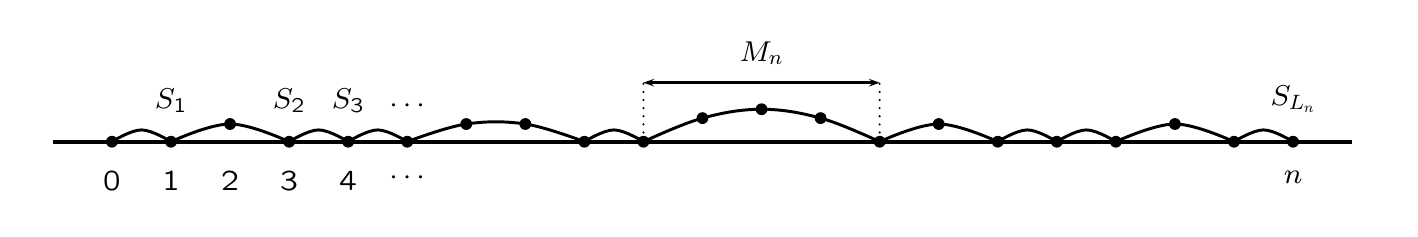}
\caption{\label{fig:1} }
\end{figure}

Denoting by $\Ex$ the expectation with respect to $\prob$ and setting
$X_a\ceq\ind_{\{a\in S\}}$ for $a\in\N_0$ in order to lighten upcoming
notations, the random pinning model for a chain of length $n\in\N$
with \emph{pinning parameter} $h\in\R$ and \emph{charges}
$\go=(\go_a)_{a\in\N_0}\in\R^{\N_0}$, or in the \emph{environment}
$\go$, is defined by the Gibbs change of measure
\begin{equation}
\frac{\dd \prob_{n,h,\go}}{\dd \prob} = \frac{\ee^{\sum_{a=1}^{n} (h+\go_a)X_a }}{Z_{n,h}(\go)} X_n\quad \text{where}\quad Z_{n,h}(\go)\ceq \Ex\big[\ee^{\sum_{a=1}^{n} (h+\go_a)X_a } X_n\big]\,.
\end{equation}
The quantity $Z_{n,h}(\go)$ is called the \emph{(quenched) partition
function}. In other words, if the $a$-th monomer from the chain
interacts with the defect line, i.e.\ $X_a=1$, then the system
receives an energetic reward $h+\go_a$. Moreover, the model is given
with pinned boundary conditions, meaning that $X_n=1$ almost surely
under $\prob_{n,h,\go}$; however, the very same results can be
established in the absence of this constraint. We shall indicate by
$\Ex_{n,h,\omega}$ the expectation with respect to $\prob_{n,h,\go}$.

\smallskip

The pinning model is explicitly solvable in the case where $\go_a=0$
for all $a\in\N_0$, a setting referred to as the \emph{pure} or
\emph{homogeneous} model (see again~\cite{G07book, Gia11, dh07}). The
model has subsequently attracted significant attention in cases where
the environment $\go$ is assumed to be random, typically sampled from
a probability space such that the charges are
i.i.d.\ (see~\cite{Ale08, AZ09, CdH13curve, GT06loc, GT06smooth,
  Lac10mart, Ton08} among others). The present paper focuses on a
correlated random environment and aims to understand the robustness of
the results established for the localized phase in the i.i.d.\ case
under a suitable exponential integrability condition (see
\cite[Assumption.~1.2]{GZ25concentration}). While it is generally
expected that short-range or fast-decaying correlations do not
substantially modify the behavior of the system, it is unclear whether
slowly decaying correlations lead to the same phenomenology or not.

\smallskip

Define $\gO\ceq \R^{\N_0}$ and let $\mathcal{F}$ be the Borel
$\sigma$-field $\mathscr{B}(\R^{\N_0})$. In the remainder of this
paper we consider a random environment $\go\in\gO$ that satisfies the
following assumption.

\begin{assumption}
  \label{assum:gauss}
The sequence $\omega=(\omega_a)_{a\in\N_0}$ is sampled from a
probability measure $\probd$ on $(\Omega,\mathcal{F})$ such that the
canonical projections $\omega\mapsto\omega_a$ form a centered,
stationary Gaussian sequence. Furthermore, the covariance function
$\gamma$ that maps $a\in\N_0$ to $\gamma(a)\ceq\int_\Omega
\omega_0\omega_a\probd[\dd\omega]$ satisfies the summability condition
$\bar\gga\ceq \sum_{a\in\N_0}|\gga(a)|<+\infty$.
\end{assumption}

Two remarks are in order. First, we stress that the function $\gamma$
satisfies the positive-semidefiniteness property
$\sum_{a=1}^n\sum_{b=1}^nu_a\gamma(|a-b|)u_b\ge 0$ for all $n\in\N$
and $u_1,\ldots,u_n\in\R$. We recall that the Kolmogorov extension
theorem guarantees that for any function $\gamma:\N_0\to\R$ enjoying
this property, together with the summability condition
$\sum_{a\in\N_0}|\gga(a)|<+\infty$, there exists a unique probability
measure $\probd$ such that $\int_\Omega
\omega_0\omega_a\probd[\dd\omega]=\gamma(a)$ for all $a\in\N_0$ and
for which Assumption \ref{assum:gauss} is met.  Second, we note that
the shift operator $\vartheta$ on $\gO$, acting as
$\gth\go=(\go_{a+1})_{a\in\N_0}$ for any $\go=(\go_a)_{a\in\N_0}$, is
a transformation that preserves the measure $\probd$ due to
stationarity.  Importantly, the hypothesis $\bar\gamma<+\infty$
implies in particular that $\lim_{a\uparrow\infty}\gamma(a)=0$, so
$\vartheta$ is not only measure-preserving but also mixing (see
\cite[§14.2, Theorem 2]{CFS82book}), and hence ergodic.

\smallskip

Classical arguments, which apply to a much broader framework than that
defined by Assumption~\ref{assum:gauss}, with ergodicity of the
environment and integrability of the charges being the only essential
ingredients (see~\cite{G07book}, notably Theorem~4.6), assure us that
the \emph{free energy}
\begin{equation*}
f(h)\ceq \lim_{n\uparrow\infty} \bbE\bigg[\frac1n \log Z_{n,h}\bigg] =  \lim_{n\uparrow\infty} \frac1n \log Z_{n,h}(\go) 
\end{equation*}
exists and is finite for all $h\in\R$, where $\Exd$ is the expectation
under $\probd$ and where the second convergence holds for $\pae$.
Furthermore, the function $f$ that maps $h$ to $f(h)$ is convex,
non-negative, and non-decreasing.  In particular, we can define the
\emph{(quenched) critical point}
\begin{equation*}
h_c\ceq\inf\big\{h\in\R: f(h)>0\big\}\in [-\infty,+\infty)\,.
\end{equation*}
In general $h_c\leq 0$ (see~\cite[Proposition~5.1]{G07book}) and a
standard comparison with the ``annealed'' model shows that, under
Assumption~\ref{assum:gauss}, $h_c\ge-\bar\gamma>-\infty$
(see~\cite{Poi13}).  The critical point $h_c$ marks the localization
phase transition: the system passes from a \emph{delocalized phase}
for $h<h_c$ to a \emph{localized phase} for $h>h_c$. In fact, the
convexity of $f$, combined with the identity $\partial_h\log
Z_{n,h}=\Ex_{n,h,\cdot}[L_n]$, implies that the average number of
contact points $\bbE[\Ex_{n,h,\cdot}[L_n]]$ satisfies
$\lim_{n\uparrow\infty}(1/n)\bbE[\Ex_{n,h,\cdot}[L_n]]=0$ if $h<h_c$
and $\liminf_{n\uparrow\infty}(1/n)\bbE[\Ex_{n,h,\cdot}[L_n]]>0$ if
$h>h_c$.

\begin{remark}
{\rm When the environment is sampled from a Gaussian law with
  non-summable correlations, i.e. $\bar\gga=+\infty$, the model
  becomes somewhat degenerate: the localization transition may
  completely vanish, indeed $h_c=-\infty$~\cite{Ber13}, and the
  ``annealed'' version of the model is
  ill-defined~\cite{Poi13}. Therefore we exclude this case from the
  present paper, and we do not expect our results to extend without
  modification to that setting.  As a matter of fact, non summable
  correlations generate a phenomenon similar to that observed in the
  i.i.d.\ case with disorder that does not satisfy the exponential
  integrability condition in
  \cite[Proposition.~1.1]{GZ25concentration}, where localization is
  driven by very attracting sites (as opposed very attracting regions
  in the correlated case).}
\end{remark}

\subsection{Main results}

We study the properties of the pinning model in its localized phase,
i.e.\ for $h>h_c$. This was recently addressed in the i.i.d.\ setting
(for a non-Gaussian environment with sub-exponential law) by the first
and last authors \cite{GZ25concentration}, who revisited and strongly
generalized~\cite{GT06loc}. They proved in particular that throughout
the localized phase
\begin{itemize}[label=--]
\item the free energy $f$ is smooth,
\item the largest gap between contact points $M_n$ scales like $\log n$ as $n\uparrow\infty$,
\item both the number of contact points $L_n$ and its mean
  $\Ex_{n,h,\cdot}[L_n]$ satisfy a Central Limit Theorem (CLT), more
  precisely they are asymptotically Gaussian with fluctuations of order
  $\sqrt{n}$.
\end{itemize}
Those results all stem from the fact that \emph{correlations decay
exponentially fast} in the localized phase under the Gibbs
measure. The present paper demonstrates that, for an environment that
satisfies Assumption~\ref{assum:gauss} (which remains in force
throughout the paper), this exponential decay persists despite the
fact that charges' correlations may decay with the distance much
slower than exponentially. We refrain from detailing the specific
statements regarding correlations decay in this section, as they are
somewhat technical; however, the reader may refer to
Lemma~\ref{lem:supercrit:polymercorrelations} and
Corollary~\ref{corol:correlations:polym} below, or to Section~3
in~\cite{GZ25concentration}, for further details.  With this
exponential decay established, we prove that all the aforementioned
results also hold in our setting.

\smallskip

We first state the smoothness of the free energy $f$ and the quenched
CLT for the number of contacts $L_n$, together with a corresponding
concentration inequality. These results are identical to Theorem~1.1
and Theorem~1.2 in~\cite{GZ25concentration}, respectively, which
pertain to the i.i.d.\ setting.

\begin{theorem}
  \label{thm:fsmooth}
  The free energy $f$ is strictly convex and $\cC^\infty$ on
  $(h_c,+\infty)$, and the following property holds for $\pae$: for
  every compact set $H\subset (h_c,+\infty)$ and $r\in\N_0$
  \begin{equation*}
    \adjustlimits\lim_{n\uparrow\infty}\sup_{h\in H}\bigg|\frac{1}{n}\partial^r_h\log Z_{n,h}(\omega)-\partial_h^r f(h)\bigg|=
    \adjustlimits\lim_{n\uparrow\infty}\sup_{h\in H}\bigg|\Exd\bigg[\frac{1}{n}\partial^r_h\log Z_{n,h}\bigg]-\partial_h^r f(h)\bigg|=0\,.
\end{equation*}
Moreover, $f$ is of class Gevrey-$3$ on $(h_c,+\infty)$: in fact for
every closed set $H\subset (h_c,+\infty)$ there exists $c>0$ such that
for all $r\in\N$
 \begin{equation*}
 \sup_{h \in H} \big| \partial_h^r f(h)\big| \le c^r (r!)^3\,.
 \end{equation*}
\end{theorem}

In light of Theorem~\ref{thm:fsmooth}, the cumulant method for normal
approximation directly yields the quenched concentration inequality
and CLT.

\begin{theorem}
  \label{thm:CLT:Ln:quenched}
  The following properties hold for $\pae$:
  \begin{enumerate}[label=({\roman*})]
\item for every closed set $H\subset (h_c,+\infty)$ there exists
  $\kappa_\omega>0$ such that for all $u\ge 0$ and $n\in\N$
 \begin{equation*}
   \sup_{h\in H}\prob_{n,h,\omega}\Big[\big|L_n-\Ex_{n,h,\omega}[L_n]\big|>u\Big] \le 2\exp\left\{-\kappa_\omega\frac{u^2}{n+u^{5/3}}\right\};
 \end{equation*}
    
 \item for every compact set $H\subset (h_c,+\infty)$
 \begin{equation*}
 \adjustlimits\lim_{n\uparrow\infty}\sup_{h\in H}\,\sup_{u\in\Rl}\,\Bigg|\prob_{n,h,\omega}\bigg[\frac{L_n-\Ex_{n,h,\omega}[L_n]}{\sqrt{n v(h)}}\le u\bigg]-
     \frac{1}{\sqrt{2\pi}}\int_{-\infty}^u\ee^{-\frac{1}{2}z^2}\dd z\Bigg|=0\,,
\end{equation*}
where $v(h)\ceq\partial_h^2 f(h)>0$.

 \end{enumerate}
\end{theorem} 

Regarding the behavior of the random centering term
$\Ex_{n,h,\cdot}[L_n]$ in parts $(i)$ and $(ii)$ of Theorem
\ref{thm:CLT:Ln:quenched}, we state the following results, which
include a CLT. They are analogous to parts $(ii)$, $(iii)$, and $(iv)$
of Theorem~1.3 in~\cite{GZ25concentration} for the i.i.d.\ case.  


\begin{theorem}
  \label{th:centering}
  The following conclusions hold:
  \begin{enumerate}[label=({\roman*})]
 
  \item for every closed set $H\subset (h_c,+\infty)$ there exists $c>0$ such that for all $n\in\N$
\begin{equation*}
  \sup_{h\in H}\Big|\Exd\big[\Ex_{n,h,\cdot}[L_n]\big]-\rho(h)n\Big|\le c\,,
\end{equation*}
where $\rho(h):=\partial_hf(h)$;

\item  there exists a locally Lipschitz continuous function $w:(h_c,+\infty)\to
[0,+\infty)$ such that for every compact set $H\subset (h_c,+\infty)$
   \begin{equation*}
     \adjustlimits\lim_{n\uparrow\infty}\sup_{h\in H}
     \Bigg|\Exd\Bigg[\bigg(\frac{\Ex_{n,h,\cdot}[L_n]-\Exd[\Ex_{n,h,\cdot}[L_n]]}{\sqrt{n}}\bigg)^{\!\!2}\Bigg]-w(h)\Bigg|=0\,;
   \end{equation*}
   
\item for every compact set $H\subset (h_c,+\infty)$ such that
  $w(h)>0$ for all $h\in H$ one has
\begin{equation*}
\adjustlimits\lim_{n\uparrow\infty}\sup_{h\in H}\,\sup_{u\in\Rl}\,\Bigg|\probd\bigg[\frac{\Ex_{n,h,\cdot}[L_n]-\Exd[\Ex_{n,h,\cdot}[L_n]]}{\sqrt{n w(h)}}\le u\bigg]
 -\frac{1}{\sqrt{2\pi}}\int_{-\infty}^u\ee^{-\frac{1}{2}z^2}\dd z\Bigg|=0\,;
\end{equation*}

 \item for all $h>h_c$
   \begin{equation*}
     w(h)\ge\bigg[\gamma(0)+2\sum_{a\in\N}\gamma(a)\bigg] v(h)^2\,,
     \end{equation*}
where $v(h)\ceq\partial_h^2 f(h)>0$. Hence, $w(h)>0$ for all $h>h_c$ when
$\gamma(0)+2\sum_{a\in\N}\gamma(a)>0$.

  \end{enumerate}
\end{theorem}

In order to illustrate the physical content of part $(iv)$ of
Theorem~\ref{th:centering} we observe that, for any $n\in\N$, the
total charge $\sum_{a=1}^n\omega_a$ has variance $\int_\Omega
\big(\sum_{a=1}^n\omega_a\big)^2\,\probd[\dd\omega]=\sum_{a=1}^n\sum_{b=1}^n\gamma(|a-b|)=n\gamma(0)+2\sum_{a=1}^n(n-a)\gamma(a)\ge
0$. Dividing by $n$, recalling that $\bar\gamma<+\infty$, and letting
$n\uparrow\infty$, the dominated convergence theorem allows us to
conclude that $\gamma(0)+2\sum_{a\in\N}\gamma(a)\ge 0$ is the limiting
scaled variance of the total charge, which therefore exhibits
non-negligible fluctuations when
$\gamma(0)+2\sum_{a\in\N}\gamma(a)>0$. Part $(iv)$ of
Theorem~\ref{th:centering} states that $w(h)>0$ for all $h>h_c$ as
soon as the total charge is non-vanishing.

Note that~\cite{GZ25concentration} proves that the centering term
$\Ex_{n,h,\cdot}[L_n]$ satisfies a concentration inequality and a CLT
in the i.i.d.\ case, again employing the cumulant method. However,
certain limitations arise in the use of this method within the
correlated framework, as discussed in further detail below. To
overcome these limitations, we establish the CLT in part $(iii)$ of
Theorem~\ref{th:centering} by resorting to the Bernstein blocking
technique \cite{IL71book}, which consists in approaching the
system as a union of large blocks of sites separated by smaller, yet
still large, gaps.

\smallskip

Before stating our last result on the largest excursion, we need to
introduce an \emph{alternative free energy}. It is the function $\mu$
defined by the following proposition, which is analogous
to~\cite[Proposition~1.3]{GZ25concentration} for the i.i.d.\ case.
Here $Z_{n,h}^{\textbf{-}}(\go)\ceq e^{-(h+\go_n)} Z_{n,h}(\go)$.
\begin{proposition}
  \label{prop:mu}
For every $h\in\R$ the limit
\begin{equation*}
\mu(h)\ceq \lim_{n\uparrow\infty}-\frac1n \log \bbE\bigg[\frac1{Z_{n,h}^{\textbf{-}}}\bigg]
\end{equation*}
exists and satisfies
\begin{equation*}
\min\bigg\{\frac{f(h)}{3},\frac{f(h)^2}{9\bar\gga}\bigg\} \leq \mu(h)\leq f(h)\,.
\end{equation*}
The function $\mu$ that maps $h$ to $\mu(h)$ is a $1$-Lipschitz continuous function on $\R$. 
\end{proposition}

The alternative free energy $\mu$ plays a pivotal role in our
proofs. In particular, the bounds in Proposition \ref{prop:mu} show
that $\mu(h)>0$ if and only if $f(h)>0$; in other words, for the
purpose of identifying $h_c$, $\mu$ is equivalent to the standard free
energy $f$.  We refer to~\cite{GZ25concentration}, as well
as~\cite{CGT12, GT06loc, GT09}, for a more detailed discussion on the
alternative free energy, including an explanation of why $1/\mu(h)$
and $1/f(h)$ represent two natural correlation lengths for the model.

The size of the largest excursion $M_n$ made by the polymer in the
localized phase $h>h_c$ is determined by $\mu(h)>0$.  This is
formalized by the following proposition, which is the analogue
of~\cite[Proposition~1.4]{GZ25concentration} for our setting.

\begin{proposition}
   \label{prop:largestjump}
    The following property holds for $\pae$: for every $h>h_c$
  \begin{equation*}
 \lim_{n\uparrow\infty}\,\frac{M_n}{\log n}=\frac{1}{\mu(h)}
  \end{equation*}
  in probability under $\prob_{n,h,\omega}$.
\end{proposition}

\begin{remark}
{\rm Within our framework, the alternative free energy could also be
  defined using the standard partition function $Z_{n,h}$ instead of
  $Z_{n,h}^{\textbf{-}}$. However, the latter is required
  in~\cite{GZ25concentration} to handle environments for which
  $\bbE[1/Z_{n,h}]=+\infty$. While this integrability issue does not
  arise in the Gaussian case, we maintain the use of $Z_{n,h}^{\textbf{-}}$
  both to remain consistent with the notation
  in~\cite{GZ25concentration} and because it allows for more
  streamlined arguments in certain proofs.}
\end{remark}

\subsection{Discussion of the results}
\noindent {\it The pinning model in correlated environment.}  The
present paper is not the first to investigate the pinning model in a
random correlated environment. We have already mentioned previous
works~\cite{Ber13, BP15, Poi13finite, Poi13} addressing the case of a
correlated Gaussian environment (see also~\cite{BT13} for the hierarchical pinning model). Moreover, the model also has been
studied in the context of specific non-Gaussian environments, such as
those derived from a renewal process or a random walk~\cite{Ber14,
  BL12, BL25I, BS10}. These works have primarily focused on the
critical properties of the pinning model at the localization phase
transition, in particular verifying the validity of the
Weinrib-Halperin prediction for disorder relevance and irrelevance
(see~\cite{WH83}). In contrast to this literature, we take a different
direction by analyzing the interior of the localized phase, aiming to
understand the pathwise behavior of the model and the fine properties
of the free energy.

\medskip

\noindent {\it Main interest of the paper.} The main contribution of
this paper is the proof that, under Assumption~\ref{assum:gauss},
correlations with respect to the Gibbs measure decay exponentially
fast within the localized phase.  With this exponential decay in hand,
we then show that almost all the aforementioned results follow by
arguing as in \cite{GZ25concentration} for the i.i.d.\ case. In
particular, the cumulant estimates and the Birkhoff-sum approach
developed therein extend to the correlated setting, with the latter
relying simply on the ergodicity of the shift operator $\vartheta$ on
$\Omega$. We refer to~\cite{GZ25concentration} for a broader
discussion of these methods and the implications of the results above,
such as for the analysis of Griffiths' singularities and the effect of
disorder on the big-jump phenomenon behind the large deviations of the
contact number in the homogeneous model. Regarding big-jump phenomena
specifically, we believe that the findings of \cite{GZ25BJ} on the
ability of disorder to suppress them can be extended to a Gaussian
correlated environment, and we plan to explore this issue in a
forthcoming work.  For the remainder of this discussion we focus on
the exponential decay of correlations under the Gibbs measure.

In order to establish the exponential decay of correlations under the
Gibbs measure we develop the following two primary tools, which
provide decorrelation inequalities under the environment law
$\bbP$:
\begin{itemize}[label=--]
\item
 a \emph{truncation argument} on the covariance function
 $\gamma$. Informally, when comparing finite-volume alternative free
 energies for two different Gaussian laws with covariance functions
 $\gamma$ and $\gamma'$, their difference can be bounded in terms of
 $\gamma-\gamma'$ (see Lemma~\ref{lem:comparison} below). This allows
 us to approximate the original Gaussian field identified by $\gamma$
 with one having finite-range dependencies, i.e.\ such that
 $\gga'(a)=0$ for all sufficiently large $a$;

\item 
a general decorrelation inequality for disorder in the spirit of
\emph{Nelson's Gaussian hypercontractivity}~\cite{Nel73}. While more
general than the truncation argument, this tool requires a quantified
positive-definiteness property for the covariance function $\gamma$
(see Proposition~\ref{prop:hypercontract} below), which would
translate into stronger assumptions on $\gamma$. However, we present
an argument that circumvents this issue by applying the
hypercontractivity estimate to the environment perturbed by the
addition of small independent charges.

  \end{itemize}

An hypothesis that has frequently appeared in the study of the pinning
model in Gaussian environment, often for technical reasons (see
e.g.~\cite{Ber13,BP15}), is the invertibility of the covariance
operator $\Gamma$ on $\ell^1(\N)$, which has entry $\gamma(|a-b|)$ in
row $a$ and column $b$. We emphasize that our results are free from
this hypothesis, requiring only the summability of correlations.

 \medskip

\noindent {\it The random centering term.}  Comparing our work with
the i.i.d.\ case studied in~\cite{GZ25concentration}, we see that the
investigation of the centering term $\Ex_{n,h,\cdot}[L_n]$ requires
new tools. Specifically, \cite{GZ25concentration} establishes a
concentration inequality and a CLT for this term by applying a
cumulant method to the random variable $\Ex_{n,h,\cdot}[L_n]$. In our
setting, this approach fails unless one assumes a very rapid decay of
the covariance function $\gamma$ (e.g. exponential), as further
illustrated in Remark~\ref{rm:correlations:envt} below. Under such a
strong assumption one would recover the same cumulant estimates as
in~\cite{GZ25concentration}, and consequently the same concentration
bound and CLT.  However, no additional hypothesis on $\gamma$ beyond
Assumption~\ref{assum:gauss} is necessary for the CLT; indeed we are
able to prove a CLT for $\Ex_{n,h,\cdot}[L_n]$ by employing the
Bernstein blocking technique.

As a last remark, we mention that part $(iv)$ of
Theorem~\ref{th:centering} guarantees that $w$ is strictly positive
provided that $\gamma(0)+2\sum_{a\in\N}\gamma(a)>0$; physically, this
means that the limiting total charge is non-vanishing. While this
condition on the covariance function $\gamma$ is very mild (recalling
that $\gamma(0)+2\sum_{a\in\N}\gamma(a)\ge 0$ holds for any $\gamma$),
it may not yet be optimal. Indeed, we cannot currently rule out the
existence of a covariance function $\gamma$ for which
$\gamma(0)+2\sum_{a\in\N}\gamma(a)=0$ and $w(h)>0$ for some, or even
all, $h>h_c$.

\subsubsection*{Organization of the paper and overview of the proofs}
Section~\ref{sec:tools} provides three general results used to prove
our main findings: a concentration inequality for the finite-volume
free energy (Proposition~\ref{prop:concentration}), a comparison of
partition functions across different environments
(Lemma~\ref{lem:comparison}), and Gaussian hypercontractivity
(Proposition~\ref{prop:hypercontract}).  In Section~\ref{sec:mu} we
prove Proposition~\ref{prop:mu} regarding the alternative free
energy. We obtain bounds on this quantity in Lemma~\ref{lem:mu:bounds}
by appealing to the aforementioned concentration inequality; we then
demonstrate existence by means of a truncation argument on the
covariance function. This argument relies on the comparison of
partition functions and the existence result for environments with
finite-range dependencies established in
Lemma~\ref{lem:mu:convergencesupercrit}. In
Section~\ref{sec:quenchedcorrelations}, specifically
Lemma~\ref{lem:supercrit:polymercorrelations}, we state the
exponential decorrelation of the polymer under the Gibbs measure, the
proof of which exploits the positivity of the alternative free energy
and the same truncation argument used in the previous section.  Then,
in Section~\ref{sec:proofmainresults} we demonstrate
Theorems~\ref{thm:fsmooth} and~\ref{thm:CLT:Ln:quenched}, as well as
Proposition~\ref{prop:largestjump}. While these proofs mostly follow
from Lemma~\ref{lem:supercrit:polymercorrelations}, the proof of
Proposition~\ref{prop:largestjump} additionally requires the Gaussian
hypercontractivity property. Finally, Section~\ref{sec:centeringterm}
is devoted to the proof of Theorem~\ref{th:centering} on the centering
term via classical Gaussian methods and the Bernstein blocking
technique for the CLT.  We emphasize that, in addition to the
ingredients above, quenched computations are central to many of our
proofs; since they remain independent of the law of the environment,
we adopt these from \cite{GZ25concentration}, omitting the details for
the sake of conciseness.

\section{Dealing with a correlated Gaussian environment}
\label{sec:tools}

In this section we provide three fundamental results that follow
directly from the Gaussian nature of the environment. These include a
concentration inequality for the finite-volume free energy of the
pinning model and a comparison of partition functions across different
environments. Furthermore, we establish Gaussian hypercontractivity
under an additional assumption regarding the environment: a quantified
positive-definiteness property of the covariance function
$\gamma$. These results will serve later as the necessary groundwork
for proving our main findings.

\smallskip

A few notions from the theory of matrices are needed in the following.
For an $n\in\N$ we denote by $\|u\|\ceq\sqrt{\sum_{a=1}^nu_a^2}$ the
Euclidean norm of $u=(u_1,\ldots,u_n)\in\R^n$ and by
$\|M\|\ceq\sup_{u\in\mathbb{S}_n}\|Mu\|$ the induced norm of a matrix
$M\in\Rl^{n\times n}$, where $\mathbb{S}_n:=\{u\in\R^n:\|u\|=1\}$ is
the unit sphere in $\R^n$. We recall that $\|M^\top\|=\|M\|$ and
$\|MM^\top\|=\|M\|^2$, where $M^\top$ is the transpose of
$M$. Moreover
\begin{equation}
  \|M\|=\sup_{u\in\mathbb{S}_n}\sup_{v\in\mathbb{S}_n}\sum_{a=1}^n\sum_{b=1}^n u_a(M)_{ab}v_b\le\sqrt{\max_{a\in\{1,\ldots,n\}}\sum_{b=1}^n|(M)_{ab}|\max_{b\in\{1,\ldots,n\}}\sum_{a=1}^n|(M)_{ab}|}\,.
\label{eq:bound_spectral_norm}
\end{equation}

\subsection{A concentration inequality}

The following concentration inequality, which will be used to prove
Proposition~\ref{prop:mu}, is a version of a well-known concentration
inequality for Lipschitz functions of Gaussian random
variables. Recall that
$\bar\gamma\ceq\sum_{a\in\N_0}|\gamma(a)|<+\infty$ by Assumption
\ref{assum:gauss}.  Let us briefly mention that, in the case of an
i.i.d.\ environment, a (slightly different) concentration inequality
was proved in~\cite{GZ25concentration} for general sub-exponential
distributions by using a variant of McDiarmid's inequality
from~\cite{MP21}.

\begin{proposition}
  \label{prop:concentration}
For all $n\in\N$, $h\in\R$, and $\eps\ge0$ 
\begin{equation*}
\bbP\Big[\big|\log Z_{n,h} - \bbE\left[\log Z_{n,h}\right]\big|>\eps\Big] \le 2\ee^{-\frac{\eps^2}{4n \bar\gga}}\,,
\end{equation*}
and the same bound holds if $Z_{n,h}$ is replaced with
$Z_{n,h}^{\textbf{-}}$.
\end{proposition}

\begin{proof}[Proof of Proposition~\ref{prop:concentration}]
Fix $n\in\N$ and $h\in\R$. We prove the statement for the
finite-volume free energy $\log Z_{n,h}$; the reader can check
that the same argument yields the result for $\log
Z_{n,h}^{\textbf{-}}$ \emph{mutatis mutandis}.  Let
$\Gamma\in\R^{n\times n}$ be the covariance matrix of $n$ consecutive
charges, with entry $\gamma(|a-b|)$ in row $a$ and column $b$, and let
$M$ be its Cholesky factor, so $\Gamma=MM^\top$. Consider the smooth
function $\phi$ that maps $u\in\R^n$ to
\begin{equation*}
  \phi(u)\ceq\log\Ex\Big[\ee^{\sum_{a=1}^n\{h+(Mu)_a\}X_a}X_n\Big]\,.
  \end{equation*}
We note that if $u_1,\ldots,u_n$ are the outcomes of $n$ independent
standard normal random variables, then $(Mu)_1,\ldots,(Mu)_n$ are $n$
Gaussian charges with covariance matrix $\Gamma$, and this implies
that the law of $u=(u_1,\ldots,u_n)\mapsto \phi(u)$ coincides with the
law of $\omega=(\omega_a)_{a\in\N_0}\mapsto\log
Z_{n,h}(\omega)=\log\Ex[\ee^{\sum_{a=1}^n\{h+\omega_a\}X_a}X_n]$.  The
proposition then follows from the concentration inequality for
Lipschitz functions of independent standard normal random variables
(see e.g.~\cite[Theorem~2.1.12]{Tao12}), provided that we establish
that $\phi$ is $\sqrt{2 n\bar\gga}$-Lipschitz with respect to the
Euclidean norm.

Given $u,v\in\R^n$, we can write $|\phi(u)-\phi(v)|\le\int_0^1|\partial_z
\phi(zu+v-zv)|\dd z$ with
\begin{align*}
  \Big|\partial_z \phi\big(zu+v-zv\big)\Big|&=
\Bigg|\sum_{a=1}^n\frac{\Ex\big[ X_a\ee^{\sum_{b=1}^n\{h+(zMu+Mv-zMv)_b\}X_b}X_n\big]}{\Ex\big[\ee^{\sum_{b=1}^n\{h+(zMu+Mv-zMv)_b\}X_b}X_n\big]}(Mu-Mv)_a\Bigg|\\
  &\le\sum_{a=1}^n \big|(Mu-Mv)_a\big|\le\sqrt{n}\,\|M(u-v)\|\le\sqrt{n}\,\|M\|\|u-v\|\,,
\end{align*}
where the second inequality is due to the Cauchy-Schwarz inequality.
On the other hand, $\|M\|=\sqrt{\|\Gamma\|}$ and
\eqref{eq:bound_spectral_norm}, combined with the symmetry of
$\Gamma$, yields
\begin{equation*}
  \|\Gamma\|\le\max_{a\in\{1,\ldots,n\}}\sum_{b=1}^n\big|\gamma(|a-b|)\big|\le 2\sum_{a\in\N_0}|\gamma(a)|=2\bar\gamma\,.
\end{equation*}
Putting everything together we conclude that $|\phi(u)-\phi(v)|\le
\sqrt{2 n\bar\gga}\,\|u-v\|$, which shows that $\phi$ is $\sqrt{2
  n\bar\gga}$-Lipschitz with respect to the Euclidean norm.
\end{proof}

\subsection{A comparison result}

The following technical lemma establishes the comparison mentioned
above of partition functions across different environments.  It will
be invoked several times throughout this paper.

\begin{lemma}
  \label{lem:comparison}
Let $\probd$ and $\probd'$ be respectively the laws of Gaussian
charges with covariance functions $\gamma$ and $\gamma'$ as in
Assumption \ref{assum:gauss} and let $\Exd$ and $\Exd'$ denote their
expectations. Then for any $h\in\R$ and integers $0\leq a_1<b_1\leq
a_2<\cdots<b_l\le n$ with $l\ge 1$
\begin{equation*}
  \Bigg|\frac{1}{n}\log \Exd \bigg[\prod_{i=1}^l\frac {1}{Z_{b_i-a_i,h}^{\textbf{-}}(\vartheta^{a_i}\cdot)}\bigg]-
  \frac{1}{n}\log \Exd' \bigg[\prod_{i=1}^l\frac {1}{Z_{b_i-a_i,h}^{\textbf{-}}(\vartheta^{a_i}\cdot)}\bigg]\Bigg|
   \le 2\sum_{c\in\N_0}\big|\gamma(c)-\gamma'(c)\big|\,.
\end{equation*}
\end{lemma}

\begin{proof}[Proof of Lemma \ref{lem:comparison}]
Fix $h\in\R$ and integers $0\leq a_1<b_1\leq a_2<\cdots<b_l\le n$ with
$l\ge 1$.  In order to lighten upcoming formulas we introduce the
random variable $\Lambda$ that maps
$\omega=(\omega_a)_{a\in\N_0}\in\Omega$ to
  \begin{equation*}
    \Lambda(\omega)\ceq\prod_{i=1}^l\frac {1}{Z_{b_i-a_i,h}^{\textbf{-}}(\vartheta^{a_i}\omega)}=\prod_{i=1}^l\frac{1}{\Ex\big[\ee^{\sum_{c=a_i+1}^{b_i-1}(h+\omega_c)X_{c-a_i}}X_{b_i-a_i}\big]}\,.
  \end{equation*}
The function $\Lambda$ actually depends only on the charges at certain
sites among $1,\ldots,n$ and it is smooth and bounded, along with its
derivatives of all order. We observe in particular that for all
$\omega=(\omega_c)_{c\in\N_0}\in\Omega$
\begin{align*}
  \partial_{\omega_a}\partial_{\omega_b}\Lambda(\omega)&=2\Lambda(\omega)\frac{\Ex\big[X_{a-a_i}\ee^{\sum_{c=a_i+1}^{b_i-1}(h+\omega_c)X_{c-a_i}}X_{b_i-a_i}\big]}{\Ex\big[\ee^{\sum_{c=a_i+1}^{b_i-1}(h+\omega_c)X_{c-a_i}}X_{b_i-a_i}\big]}
  \frac{\Ex\big[X_{b-a_i}\ee^{\sum_{c=a_i+1}^{b_i-1}(h+\omega_c)X_{c-a_i}}X_{b_i-a_i}\big]}{\Ex\big[\ee^{\sum_{c=a_i+1}^{b_i-1}(h+\omega_c)X_{c-a_i}}X_{b_i-a_i}\big]}\\
  &\quad-\Lambda(\omega)\frac{\Ex\big[X_{a-a_i}X_{b-a_i}\ee^{\sum_{c=a_i+1}^{b_i-1}(h+\omega_c)X_{c-a_i}}X_{b_i-a_i}\big]}{\Ex\big[\ee^{\sum_{c=a_i+1}^{b_i-1}(h+\omega_c)X_{c-a_i}}X_{b_i-a_i}\big]}
\end{align*}
if $a$ and $b$ lay in the same block $\{a_i+1,\ldots,b_i-1\}$ and
\begin{equation*}
  \partial_{\omega_a}\partial_{\omega_b}\Lambda(\omega)=\Lambda(\omega)\frac{\Ex\big[X_{a-a_i}\ee^{\sum_{c=a_i+1}^{b_i-1}(h+\omega_c)X_{c-a_i}}X_{b_i-a_i}\big]}{\Ex\big[\ee^{\sum_{c=a_i+1}^{b_i-1}(h+\omega_c)X_{c-a_i}}X_{b_i-a_i}\big]}
  \frac{\Ex\big[X_{b-a_j}\ee^{\sum_{c=a_j+1}^{b_j-1}(h+\omega_c)X_{c-a_j}}X_{b_j-a_j}\big]}{\Ex\big[\ee^{\sum_{c=a_j+1}^{b_j-1}(h+\omega_c)X_{c-a_j}}X_{b_j-a_j}\big]}
\end{equation*}
if $a$ and $b$ lay in two different blocks $\{a_i+1,\ldots,b_i-1\}$
and $\{a_j+1,\ldots,b_j-1\}$. These formulas allow us to deduce that
$|\partial_{\omega_a}\partial_{\omega_b}\Lambda|\le 2\Lambda$ for all
$a,b\in\{1,\ldots,n\}$.

The proof of the lemma involves a standard argument of Gaussian
interpolation:
\begin{align}
  \nonumber
   \big|\log\Exd[\Lambda]-\log\Exd'[\Lambda]\big|&=
   \bigg|\log\int_\Omega\probd[\dd\omega]\,\Lambda(\omega)-\log\int_\Omega\probd'[\dd\omega']\,\Lambda(\omega')\bigg|\\
   &\le \int_0^1\dd z \Bigg|\partial_z\log\int_{\Omega}\probd[\dd\omega]\int_{\Omega}\probd'[\dd\omega']\,\Lambda\big(z\omega+\sqrt{1-z^2}\omega'\big)\Bigg|\,.
   \label{eq:comparison_1}
\end{align}
On the other hand, for $z\in(0,1)$ we find
\begin{align}
  \nonumber
  &\partial_z\int_{\Omega}\probd[\dd\omega]\int_{\Omega}\probd'[\dd\omega']\,\Lambda\big(z\omega+\sqrt{1-z^2}\omega'\big)\\
  \nonumber
  &\qquad=\int_{\Omega}\probd[\dd\omega]\int_{\Omega}\probd'[\dd\omega']\sum_{a=1}^n\bigg(\omega_a-\frac{z}{\sqrt{1-z^2}}\,\omega_a'\bigg)
  \partial_{\omega_a}\Lambda\big(z\omega+\sqrt{1-z^2}\omega'\big)\\
  &\qquad=z\sum_{a=1}^n\sum_{b=1}^n\Big[\gamma(|a-b|)-\gamma'(|a-b|)\Big]\int_{\Omega}\probd[\dd\omega]\int_{\Omega}\probd'[\dd\omega']\,
  \partial_{\omega_a}\partial_{\omega_b}\Lambda\big(z\omega+\sqrt{1-z^2}\omega'\big)\,,
\label{eq:comparison_2}
\end{align}
where the second equality is due to Gaussian integration by parts (see
\cite[Lemma~5.2]{Bis20notes}). Then, using
$|\partial_{\omega_a}\partial_{\omega_b}\Lambda|\le 2\Lambda$ for all
$a,b\in\{1,\ldots,n\}$ in \eqref{eq:comparison_2}, we conclude from
\eqref{eq:comparison_1} that
\begin{align*}
   \big|\log\Exd[\Lambda]-\log\Exd'[\Lambda]\big|\le \int_0^1\dd z\,
   2z\sum_{a=1}^n\sum_{b=1}^n\Big|\gamma(|a-b|)-\gamma'(|a-b|)\Big|\le  2n\sum_{c\in\N_0}\big|\gamma(c)-\gamma'(c)\big|\,,
\end{align*}
which concludes the proof.
\end{proof}

\subsection{The Gaussian hypercontractivity}

Finally, we introduce a general decorrelation inequality following
Nelson's hypercontractivity for Gaussian fields~\cite{Nel73}, which
will be used in the proof of Proposition~\ref{prop:largestjump}.

\begin{proposition}
  \label{prop:hypercontract}
  Suppose that there exists $\eta>0$ such that
  $\sum_{a=1}^n\sum_{b=1}^nu_a\gamma(|a-b|)u_b\ge\eta$ for all
  $n\in\N$ and $(u_1,\ldots,u_n)\in\mathbb{S}_n$.  Then for all
  $n\in\N$, integers $0\leq a_1<b_1\leq a_2<\cdots <b_l\le n$ with
  $l\ge 1$, and non-negative measurable functions
  $\phi_1:\R^{b_1-a_1}\to\R,\ldots,\phi_l:\R^{b_l-a_l}\to\R$ 
\begin{equation*}
  \int_\Omega\prod_{i=1}^l \phi_i(\go_{a_i+1},\ldots,\go_{b_i})\,\probd[\dd\omega] \le
  \prod_{i=1}^l \Bigg\{\int_\Omega\phi_i(\go_{a_i+1},\ldots,\go_{b_i})^\frac{2\bar\gamma}{\eta}\,\probd[\dd\omega]\Bigg\}^{\!\frac{\eta}{2\bar\gamma}}\,.
\end{equation*}
\end{proposition}

\begin{proof}[Proof of Proposition~\ref{prop:hypercontract}]
Given $n\in\N$ and integers $0\leq a_1<b_1\leq a_2<\cdots <b_l\le n$
with $l\ge 1$, it suffices to prove the proposition when $a_i=b_{i-1}$ for
all $i$, where $b_0\ceq0$, and $b_l=n$. Then, let $\Gamma_m$ denote
the covariance matrix of $m$ consecutive charges,
i.e.\ $(\Gamma_m)_{ab}=\gamma(|a-b|)$ for $a,b\in\{1,\ldots,m\}$, and
introduce the block diagonal matrix
\begin{equation*}
  M\ceq\begin{pmatrix}
  \Gamma_{b_1-b_0} &  &  & & \\
    &  & \ddots &   & \\
     & &  & \Gamma_{b_l-b_{l-1}} & 
  \end{pmatrix}\,.
\end{equation*}
We note that $M\in\R^{n\times n}$ and that $2\bar\gamma
M-\eta\Gamma_n$ is positive semi-definite. In fact the hypotheses on
the covariance function $\gamma$ imply
$\sum_{a=1}^n\sum_{b=1}^nu_a(M)_{ab}u_b\ge\eta$ and
$\sum_{a=1}^n\sum_{b=1}^nu_a(\Gamma)_{ab}u_b\le 2\bar\gamma$ for all
$(u_1,\ldots,u_n)\in\mathbb{S}_n$, the latter being an application
of \eqref{eq:bound_spectral_norm}.  Then we deduce from~\cite[Theorem
  1]{CDP15} that for all non-negative measurable functions
$\phi_1:\R^{b_1-b_0}\to\R,\ldots,\phi_l:\R^{b_l-b_{l-1}}\to\R$
\begin{align*}
  \int_\Omega\prod_{i=1}^l\phi_i(\omega_{b_{i-1}+1},\ldots,\omega_{b_i})\,\bbP[\dd\omega]&\le
  \Bigg\{\prod_{i=1}^l\int_\Omega \phi_i(\omega_{b_{i-1}+1},\ldots,\omega_{b_i})^{\frac{2\bar\gamma}{\eta}}\,\bbP[\dd\omega]\Bigg\}^{\!\frac{\eta}{2\bar\gamma}}\,.
  \qedhere
\end{align*}
\end{proof}

\section{The alternative free energy}
\label{sec:mu}

In this section we demonstrate Proposition~\ref{prop:mu}, whose proof
relies on the general results established in the previous section,
together with two additional more specialized lemmas. The first of the
latter states bounds on the limits superior and inferior
\begin{equation*}
  \mu_+(h)\ceq\limsup_{n\uparrow\infty}-\frac{1}{n}\log\bbE\bigg[\frac{1}{Z_{n,h}^{\textbf{-}}}\bigg]\quad\text{and}\quad
  \mu_-(h)\ceq\liminf_{n\uparrow\infty}-\frac{1}{n}\log\bbE\bigg[\frac{1}{Z_{n,h}^{\textbf{-}}}\bigg]\,.
\end{equation*}
The second one shows that they are equal when the environment has
finite range correlations: this is one of the most intricate parts of
the section.  We these lemmas at hand, we will then able to prove
Proposition~\ref{prop:mu} by means of a truncation argument.

\begin{lemma}
  \label{lem:mu:bounds}
For all $h\in\R$
\begin{equation*}
\min\bigg\{\frac{f(h)}{3},\frac{f(h)^2}{9\bar\gga}\bigg\} \leq \mu_-(h) \leq \mu_+(h) \leq f(h)\,.
\end{equation*}
\end{lemma}

\begin{proof}
The upper bound on $\mu_+(h)$ is a direct consequence of Jensen's
inequality and the fact that the charges are centered: for all
$n\in\N$ and $h\in\R$
\begin{equation*}
-\log \bbE\bigg[\frac{1}{Z_{n,h}^{\textbf{-}}}\bigg] \leq \bbE\big[\log Z_{n,h}^{\textbf{-}}\big] = \bbE\big[\log Z_{n,h}\big]-h\,.
\end{equation*}
Then, dividing by $n$ and letting $n\uparrow\infty$, this yields that
$\mu_+(h)\leq f(h)$ for any $h\in\R$.

\smallskip

Regarding the lower bound on $\mu_-(h)$, it is a consequence of the
concentration inequality from Proposition~\ref{prop:concentration}. In
fact for every $n\in\N$, $h\in\R$, and $\go\in\Omega$ one has
\begin{equation*}
\frac{p(n)}{Z_{n,h}^{\textbf{-}}(\go)} = \prob_{n,h,\go}[T_1=n] 
\leq \ind_{\big\{\log Z_{n,h}^{\textbf{-}}(\go) - \bbE[\log Z_{n,h}^{\textbf{-}}] < -\frac{2f(h)}{3}n\big\}} +  p(n)\ee^{\frac{2f(h)}{3}n-\bbE[\log Z_{n,h}^{\textbf{-}}]}\,,
\end{equation*}
so integrating with respect to $\bbP[\dd \go]$ and invoking
Proposition~\ref{prop:concentration} one obtains
\begin{align*}
  p(n)\,\bbE\bigg[\frac{1}{Z_{n,h}^{\textbf{-}}}\bigg] &\leq 2\ee^{-\frac{f(h)^2}{9\bar\gga}n} + p(n) \ee^{\frac{2f(h)}{3}n-\bbE[\log Z_{n,h}^{\textbf{-}}]}\\
  &=2\ee^{-\frac{f(h)^2}{9\bar\gga}n} + p(n) \ee^{\frac{2f(h)}{3}n-\bbE[\log Z_{n,h}]+h}\,.
\end{align*}
Taking the logarithm, dividing by $n$ and letting $n\uparrow\infty$, this finally yields
\begin{equation*}
\mu_-(h)\geq \min\bigg\{\frac{f(h)}{3},\frac{f(h)^2}{9\bar\gga}\bigg\}
\end{equation*}
since $p(n)$ decays polynomially fast as $n\uparrow\infty$.
\end{proof}

We now turn to the case of finite-range correlations.
\begin{lemma}
  \label{lem:mu:convergencesupercrit}
  Suppose that the covariance function $\gamma$ satisfies
  $\gamma(a)=0$ for all $a>r$ with some $r\in\N$. Then
  $\mu_+(h)=\mu_-(h)$ for all $h\in\Rl$.
\end{lemma}

Following \cite{GZ25concentration}, we recall that the properties of
slowly varying functions (see~\cite[Theorem~1.2.1 and
  Proposition~1.3.6]{BGT89}) imply that there exists a constant
$\xi>0$ such that for all $t,\tau\in\N$
\begin{equation}\label{eq:defxi}
\frac{p(t+\tau)}{p(t)p(\tau)} \leq \xi \min\{t^\xi, \tau^\xi\}\,.
\end{equation}
Replacing $\xi$ by a larger value if necessary, we may assume that
$\xi$ also fulfils $p(t)\ge(1+t)^{-\xi}$ for all $t\in\N$. As
explained in Remark 2.2 of \cite{GZ25concentration}, the number $\xi$
provides the bound
\begin{equation}
    \Ex_{n,h,\omega}[X_a]\ge \frac{1}{1+\xi\ee^{-h-\omega_a}\min\{a^\xi,(n-a)^\xi\}}\,,
\label{eq:local_bound}
\end{equation}
which holds for all $n\in\N$, $h\in\Rl$,
$\omega=(\omega_b)_{b\in\N_0}\in\Omega$, and
$a\in\{0,\ldots,n\}$. This bound plays a key role in the proof of
Lemma~\ref{lem:mu:convergencesupercrit}, together with the renewal
property of the pinning model.

\begin{proof}[Proof of Lemma~\ref{lem:mu:convergencesupercrit}]
When $h\leq h_c$ one has $f(h)=0$, so Lemma~\ref{lem:mu:bounds}
implies that $\mu_+(h)=\mu_-(h)=0$. Most of the work is then devoted
to the case $h>h_c$.  Fix $h>h_c$ and, in order to lighten upcoming
formulas, set for all $n\in\N$
\begin{equation}
\nu_n\ceq -\log\int_\Omega\frac{\ee^{-\omega_0}}{Z_{n,h}^{\textbf{-}}(\omega)}\,\probd[\dd\omega]\,.
\label{eq:def_nu_mu}
\end{equation}
We shall show that there exist $n_o\in\N$ and $C>0$ such that for
every $m,n\in\N$ not smaller than $n_o$
\begin{equation}
   \nu_{m+n}\ge \nu_m+\nu_n-C-C\sqrt{\min\{m,n\}}\,.
\label{eq:muconv:main}
\end{equation}
Using the bound $\sqrt{\min\{m,n\}}\le
2(\sqrt{m}+\sqrt{n}-\sqrt{m+n})$, we then realize that the sequence
$(\nu_n-C-2C\sqrt{n})_{n\in\N}$ is eventually
superadditive. Consequently, Fekete's lemma assures us that the limit
$\lim_{n\uparrow\infty}\nu_n/n\eqc q$ exists. We claim that
$\mu_+(h)\le q$ and $q\le\mu_-(h)$, from which we deduce that
$\mu_+(h)=\mu_-(h)=q$, thus proving the lemma.  These bounds are
obtained as follows.  Noting that
$Z_{n,h}^{\textbf{-}}(\omega)\ge\Ex[\ee^{\sum_{a=1}^{n-1}(h+\omega_a)X_a}\prod_{k=1}^{n-1}(1-X_k)X_n]=p(n)$,
we can write for any $n\in\N$ and $\zeta\in\R$
\begin{align*}
  \ee^{-\zeta}\int_\Omega\frac{\ee^{-\omega_0}}{Z_{n,h}^{\textbf{-}}(\omega)}\,\probd[\dd\omega]&=
  \int_\Omega\ind_{\{\omega_0>-\zeta\}}\frac{\ee^{-\zeta-\omega_0}}{Z_{n,h}^{\textbf{-}}(\omega)}\,\probd[\dd\omega]+
  \int_\Omega\ind_{\{\omega_0\le-\zeta\}}\frac{\ee^{-\zeta-\omega_0}}{Z_{n,h}^{\textbf{-}}(\omega)}\,\probd[\dd\omega]\\
  &\le\Exd\bigg[\frac{1}{Z_{n,h}^{\textbf{-}}}\bigg]+
  \frac{1}{p(n)}\int_\Omega\ind_{\{\omega_0\le-\zeta\}}\ee^{-\zeta-\omega_0}\,\probd[\dd\omega]
\end{align*}
and
\begin{align*}
  \ee^{\zeta}\int_\Omega\frac{\ee^{-\omega_0}}{Z_{n,h}^{\textbf{-}}(\omega)}\,\probd[\dd\omega]&\ge
  \int_\Omega\ind_{\{\omega_0\le \zeta\}}\frac{\ee^{\zeta-\omega_0}}{Z_{n,h}^{\textbf{-}}(\omega)}\,\probd[\dd\omega]\\
  &\ge\Exd\bigg[\frac{1}{Z_{n,h}^{\textbf{-}}}\bigg]-\frac{1}{p(n)}\int_\Omega\ind_{\{\omega_0>\zeta\}}\,\probd[\dd\omega]\,.
\end{align*}
Then, choosing $\zeta= 2\sqrt{\gamma(0)f(h)n}$, a standard Gaussian
tail estimate yields
\begin{equation}
   \ee^{-2\sqrt{\gamma(0)f(h)n}}\int_\Omega\frac{\ee^{-\omega_0}}{Z_{n,h}^{\textbf{-}}(\omega)}\,\probd[\dd\omega]\le\Exd\bigg[\frac{1}{Z_{n,h}^{\textbf{-}}}\bigg]+
  \frac{\ee^{-2f(h)n}}{p(n)}
\label{eq:mu:convergencesupercrit_0}
\end{equation}
and
\begin{equation}
  \ee^{ 2\sqrt{\gamma(0)f(h)n}}\int_\Omega\frac{\ee^{-\omega_0}}{Z_{n,h}^{\textbf{-}}(\omega)}\,\probd[\dd\omega]\ge\Exd\bigg[\frac{1}{Z_{n,h}^{\textbf{-}}}\bigg]-
  \frac{\ee^{-2f(h)n}}{p(n)}\,.
\label{eq:mu:convergencesupercrit_00}
\end{equation}
Since $\mu_+(h)\le f(h)$ by Lemma~\ref{lem:mu:bounds} and $p(n)$
decays polynomially fast as $n\uparrow\infty$, $\ee^{-2f(h)n}/p(n)$
decays faster than $\Ex[1/Z_{n,h}^{\textbf{-}}]$ as
$n\uparrow\infty$. So, taking the logarithm in
\eqref{eq:mu:convergencesupercrit_0}, dividing by $n$, and then
letting $n\uparrow\infty$, we find $\mu_+(h)\le q$. Similarly, from
\eqref{eq:mu:convergencesupercrit_00} we get $q\le\mu_-(h)$.

\smallskip

Let us now verify \eqref{eq:muconv:main}.  Jensen's inequality,
together with the fact that the charges are centered, gives for all
$n\in\N$
\begin{equation*}
-\log \int_\Omega\frac{\ee^{-\omega_0}}{Z_{n,h}^{\textbf{-}}(\omega)}\,\probd[\dd\omega] \leq \bbE\big[\log Z_{n,h}^{\textbf{-}}\big] = \bbE\big[\log Z_{n,h}\big]-h\,,
\end{equation*}
so we can state that there exists a number $n_o>r$ such that for all
integers $n\ge n_o$
\begin{equation}
(2r)^{\frac{1}{3}}\ee^{-2h+3\gamma(0)}(1+n)^\xi\ee^{-2f(h) n}\le
  \int_\Omega\frac{\ee^{-\omega_0}}{Z_{n,h}^{\textbf{-}}(\omega)}\,\probd[\dd\omega]\,.
\label{eq:tech_1}
\end{equation}
Pick $m,n\in\N$ such that $r<n_o\le m\le n$ and define
\begin{equation*}
  C\ceq \max\Big\{(r+1)|h|+\log(2+\xi r^\xi),r\sqrt{12\gamma(0)f(h)}\Big\}  \,.
\end{equation*}
Let us prove that $\nu_{m+n}\ge \nu_m+\nu_n-C-C\sqrt{m}$. Set $l\ceq m-r$
for brevity. The renewal property of the pinning model (see~\cite[Lemma~2.1]{GZ25concentration}), which we
invoke to obtain the second equality in the following chain of
equations, allows us to write down for all
$\omega=(\omega)_{a\in\N_0}\in\Omega$ the bound
\begin{align}
  \nonumber
  Z_{m+n,h}^{\textbf{-}}(\omega)&=\Ex\big[\ee^{\sum_{a=1}^{m+n-1}(h+\omega_a)X_a}X_{m+n}\big]\\
  \nonumber
  &\ge\Ex\big[\ee^{\sum_{a=1}^{m+n-1}(h+\omega_a)X_a}X_lX_mX_{m+n}\big]\\
  \nonumber
  &=\Ex\big[\ee^{\sum_{a=1}^l(h+\omega_a)X_a}X_l\big]\Ex\big[\ee^{\sum_{a=1}^{m-l}(h+\omega_{l+a})X_a}X_{m-l}\big]\Ex\big[\ee^{\sum_{a=1}^{n-1}(h+\omega_{m+a})X_a}X_n\big]\\
  &=Z_{l,h}(\omega) Z_{m-l,h}^{\textbf{-}}(\vartheta^l\omega)\ee^{h+\omega_m} Z_{n,h}^{\textbf{-}}(\vartheta^m\omega)\,.
\label{eq:mu_aux_1}
\end{align}
Using $Z_{m-l,h}^{\textbf{-}}(\vartheta^l\omega)\ge p(m-l)=p(r)$, we
can recast \eqref{eq:mu_aux_1} as
\begin{equation*}
  \frac{\ee^{-\omega_0}}{Z_{m+n,h}^{\textbf{-}}(\omega)}\le\frac{1}{p(r)}\frac{\ee^{-h-\omega_0}}{ Z_{l,h}(\omega)}
  \frac{\ee^{-\omega_m}}{Z_{n,h}^{\textbf{-}}(\vartheta^m\omega)}\,.
\label{eq:mu_aux_2}
\end{equation*}
Integrating the latter with respect to $\probd[\dd\omega]$, the fact
that the charges at sites $0,\ldots,l-1$ are independent of those at
sites $m,\ldots,m+n-1$, because $m+1-l>r$, allows us to reach the bound
\begin{equation}
  \int_\Omega\frac{\ee^{-\omega_0}}{Z_{m+n,h}^{\textbf{-}}(\omega)}\,\probd[\dd\omega]\le
  \frac{1}{p(r)}\int_\Omega\frac{\ee^{-h-\omega_0}}{ Z_{l,h}(\omega)}\,\probd[\dd\omega]
  \int_\Omega\frac{\ee^{-\omega_0}}{Z_{n,h}^{\textbf{-}}(\omega)}\,\probd[\dd\omega]\,.
\label{eq:step_1}
\end{equation}
In view of \eqref{eq:def_nu_mu}, in order to prove that $\nu_{m+n}\ge
\nu_m+\nu_n-C-C\sqrt{m}$ it remains to show that
\begin{equation}
 \frac{1}{p(r)}\int_\Omega\frac{\ee^{-h-\omega_0}}{ Z_{l,h}(\omega)}\,\probd[\dd\omega]
  \le \ee^{C+C\sqrt{m}} \int_\Omega\frac{\ee^{-\omega_0}}{Z_{m,h}^{\textbf{-}}(\omega)}\,\probd[\dd\omega]\,.
\label{eq:step_5}
\end{equation}

Starting from \eqref{eq:local_bound} and using the renewal property of
the pinning model we can state that for all
$\omega=(\omega_a)_{a\in\N_0}\in\Omega$
\begin{align*}
  \frac{1}{1+\xi\ee^{-h-\omega_l}(m-l)^\xi}\le \Ex_{m,h,\omega}[X_l]&=\frac{\Ex[\ee^{\sum_{a=1}^{m-1}(h+\omega_a)X_a}X_lX_m]}{Z_{m,h}^{\textbf{-}}(\omega)}\\
  &=\frac{\Ex[\ee^{\sum_{a=1}^l(h+\omega_a)X_a}X_l]\Ex[\ee^{\sum_{a=l+1}^{m-1}(h+\omega_a)X_{a-l}}X_{m-l}]}{Z_{m,h}^{\textbf{-}}(\omega)}\\
  &\le\frac{Z_{l,h}(\omega)\,\ee^{\sum_{a=l+1}^{m-1}(|h|+|\omega_a|)}}{Z_{m,h}^{\textbf{-}}(\omega)}\,.
\end{align*}
Then, recalling that $m-l=r$, we see that
\begin{align*}
  \frac{1}{Z_{l,h}(\omega)}\le \big(1+\xi r^\xi\ee^{-h-\omega_l}\big)\frac{\ee^{\sum_{a=l+1}^{m-1}(|h|+|\omega_a|)}}{Z_{m,h}^{\textbf{-}}(\omega)}
  &\le\big(1+\xi r^\xi\big)\ee^{r|h|+\sum_{a=l}^{m-1}|\omega_a|}\frac{1}{Z_{m,h}^{\textbf{-}}(\omega)}\,.
\end{align*}
This bound is useful only for certain values of the charges, while for
other values we observe that $Z_{l,h}(\omega)\ge
p(l)\ee^{h+\omega_l}$. So, truncating the charges we write
\begin{align*}
  \frac{1}{Z_{l,h}(\omega)}&\le
  \frac{\mathds{1}_{\{\max\{|\omega_l|,\ldots,|\omega_{m-1}|\}\le \sqrt{12\gamma(0)f(h) m}\}}}{Z_{l,h}(\omega)}
  +\frac{\mathds{1}_{\{\max\{|\omega_l|,\ldots,|\omega_{m-1}|\}>\sqrt{12\gamma(0)f(h) m}\}}}{p(l)\ee^{h+\omega_l}} \\
  &\le  \big(1+\xi r^\xi\big)\ee^{r|h|+r\sqrt{12\gamma(0)f(h) m}}\,\frac{1}{Z_{m,h}^{\textbf{-}}(\omega)}
  +\frac{\mathds{1}_{\{\max\{|\omega_l|,\ldots,|\omega_{m-1}|\}>\sqrt{12\gamma(0)f(h) m}\}}}{p(l)\ee^{h+\omega_l}}\,.
\end{align*}
Integrating this estimate with respect to $\probd[\dd\omega]$ and
invoking H\"older's inequality for three factors we thus find
\begin{align*}
  \int_\Omega\frac{\ee^{-h-\omega_0}}{ Z_{l,h}(\omega)}\,\probd[\dd\omega]&\le \big(1+\xi r^\xi\big)\ee^{(r+1)|h|+r\sqrt{12\gamma(0)f(h) m}}
  \int_\Omega\frac{\ee^{-\omega_0}}{Z_{m,h}^{\textbf{-}}(\omega)}\,\probd[\dd\omega]\\
  &\quad+\frac{\ee^{-2h}}{p(l)}\Bigg\{\bigg(\int_\Omega\ee^{-3\omega_0}\,\probd[\dd\omega]\bigg)^{\!2}
  \int_\Omega\mathds{1}_{\big\{\max\{|\omega_1|,\ldots,|\omega_r|\}>\sqrt{12\gamma(0)f(h)m}\big\}}\probd[\dd\omega]\Bigg\}^{\!\frac{1}{3}}\,.
\end{align*}
We have
$\int_\Omega\ee^{-3\omega_0}\,\probd[\dd\omega]=\ee^{\frac{9}{2}\gamma(0)}$
and, combining a union bound with standard Gaussian tail estimates, we
see that
\begin{equation*}
  \int_\Omega\ind_{\big\{\max\{|\omega_1|,\ldots,|\omega_r|\}>\sqrt{12\gamma(0)f(h)m}\big\}}\probd[\dd\omega]\le r \int_\Omega\ind_{\big\{|\omega_0|>\sqrt{12\gamma(0)f(h)m}\big\}}\probd[\dd\omega]
  \le 2r\ee^{-6f(h)m}\,.
  \end{equation*}
Putting everything together we finally observe that
\begin{align*}
  \int_\Omega\frac{\ee^{-h-\omega_0}}{ Z_{l,h}(\omega)}\,\probd[\dd\omega]&\le \big(1+\xi r^\xi\big)\ee^{(r+1)|h|+r\sqrt{12\gamma(0)f(h) m}}
  \int_\Omega\frac{\ee^{-\omega_0}}{Z_{m,h}^{\textbf{-}}(\omega)}\,\probd[\dd\omega]\\
  &\qquad\qquad\qquad\qquad\qquad\qquad\qquad\qquad\quad+\frac{(2r)^{\frac{1}{3}}\ee^{-2h+3\gamma(0)}}{p(l)}\,\ee^{-2f(h) m}
\end{align*}
with
\begin{equation}
  \frac{(2r)^{\frac{1}{3}}\ee^{-2h+3\gamma(0)}}{p(l)}\,\ee^{-2f(h) m}\le (2r)^{\frac{1}{3}}\ee^{-2h+3\gamma(0)}(1+m)^\xi\ee^{-2f(h) m}\le
  \int_\Omega\frac{\ee^{-\omega_0}}{Z_{m,h}^{\textbf{-}}(\omega)}\,\probd[\dd\omega]
\label{eq:step_8}
\end{equation}
thanks to $p(t)\ge(1+t)^{-\xi}$ for all $t\in\N$ and to
\eqref{eq:tech_1}. This demonstrates that
\begin{align*}
  \int_\Omega\frac{\ee^{-h-\omega_0}}{ Z_{l,h}(\omega)}\,\probd[\dd\omega]&\le \big(2+\xi r^\xi\big)\ee^{(r+1)|h|+r\sqrt{12\gamma(0)f(h) m}}
  \int_\Omega\frac{\ee^{-\omega_0}}{Z_{m,h}^{\textbf{-}}(\omega)}\,\probd[\dd\omega]\,,
\end{align*}
which in turn implies \eqref{eq:step_5} because of our choice of the constant $C$.
\end{proof}

With Lemmas~\ref{lem:mu:bounds},~\ref{lem:mu:convergencesupercrit},
and~\ref{lem:comparison} at hand, we are now ready to prove
Proposition~\ref{prop:mu} by a truncation argument.

\begin{proof}[Proof of Proposition~\ref{prop:mu}]
Let $\probd$ be an arbitrary probability measure on
$(\Omega,\mathcal{F})$ that satisfies
Assumption~\ref{assum:gauss}. For $r\in\N$ define the real function
$\gamma_r$ on $\N_0$ by
$\gamma_r(0)\ceq\gamma(0)+2\sum_{a=r+1}^\infty|\gamma(a)|$ and
$\gamma_r(a)\ceq\ind_{\{a\le r\}}\gamma(a)$ for $a\in\N$, where
$\gamma$ is the covariance function associated with $\probd$. We claim
that $\gamma_r$ is a covariance function,
i.e.\ $\sum_{a=1}^n\sum_{b=1}^n u_a\gamma_r(|a-b|)u_b\ge 0$ for all
$n\in\N$ and $u=(u_1,\ldots,u_n)\in\Rl^n$, so in particular there
exists a law $\probd_r$ on $(\Omega,\mathcal{F})$ with covariance
function $\gamma_r$ that satisfies Assumption~\ref{assum:gauss}. Indeed, recalling that $\sum_{a=1}^n\sum_{b=1}^n u_a\gamma(|a-b|)u_b\ge
0$ and letting $M\in\Rl^{n\times n}$ be the symmetric matrix with
entry $\ind_{\{|a-b|>r\}}\gamma(|a-b|)$ in row $a$ and column $b$, an
application of \eqref{eq:bound_spectral_norm} to $M$ shows that
\begin{align*}
  \sum_{a=1}^n\sum_{b=1}^n u_a\gamma_r(|a-b|)u_b&=2\sum_{a=r+1}^\infty|\gamma(a)|\,\|u\|^2+\sum_{a=1}^n\sum_{b=1}^n u_a\gamma(|a-b|)u_b-\sum_{a=1}^n\sum_{b=1}^n u_a(M)_{ab}u_b\\
  &\ge\bigg[2\sum_{a=r+1}^\infty|\gamma(a)|-\max_{a\in\{1,\ldots,n\}}\sum_{b=1}^n|(M)_{ab}|\bigg]\,\|u\|^2\,.
\end{align*}
On the other hand, we have
$\max_{a\in\{1,\ldots,n\}}\sum_{b=1}^n|(M)_{ab}|\le2\sum_{a=r+1}^\infty|\gamma(a)|$, so this implies that $\gamma_r$ is a covariance function.

\smallskip

Fixing $h\in\R$ and denoting by $\Exd_r$ the expectation under
$\probd_r$, Lemma~\ref{lem:mu:convergencesupercrit} assures us that
for all $r\in\N$
\begin{equation*}
  \limsup_{n \uparrow \infty} -\frac 1n \log \Exd_r \bigg[\frac {1}{Z_{n,h}^{\textbf{-}}}\bigg] =
  \liminf_{n \uparrow \infty} -\frac 1n \log \Exd_r \bigg[\frac {1}{Z_{n,h}^{\textbf{-}}}\bigg]\eqc\mu_r(h)\,.
\end{equation*}
Consequently, for $r,r'\in\N$ we have $|\mu_r(h)-\mu_{r'}(h)|\le
6\sum_{a=\min\{r,r'\}+1}^\infty|\gamma(a)|$ according to Lemma
\ref{lem:comparison}, which in turn shows that the sequence with
$r$-term $\mu_r(h)$ is a Cauchy sequence, so the limit
$\lim_{r\uparrow\infty}\mu_r(h)\eqc\mu_\infty(h)$ exists. Lemma
\ref{lem:comparison} also yields
\begin{equation}
  \label{prop:mu_1}
  \Bigg|\frac 1n \log \Exd \bigg[\frac {1}{Z_{n,h}^{\textbf{-}}}\bigg]-\frac 1n \log \Exd_r \bigg[\frac {1}{Z_{n,h}^{\textbf{-}}}\bigg]\Bigg|
  \le 6\sum_{a=r+1}^\infty|\gamma(a)|
\end{equation}
for all $n,r\in\N$ and $h\in\Rl$. Letting first $n\uparrow\infty$ and
then $r\uparrow\infty$ in \eqref{prop:mu_1}, we finally realize that
\begin{equation*}
  \mu_\infty(h)\le\liminf_{n\uparrow\infty}-\frac 1n \log \Exd \bigg[\frac {1}{Z_{n,h}^{\textbf{-}}}\bigg]\le\limsup_{n\uparrow\infty}
  -\frac 1n \log \Exd \bigg[\frac {1}{Z_{n,h}^{\textbf{-}}}\bigg]\le\mu_\infty(h)\,,
\end{equation*}
which proves that the limit
$\mu(h)\ceq\lim_{n\uparrow\infty}-(1/n)\log
\bbE[1/Z_{n,h}^{\textbf{-}}]$ exists. Furthermore, one has $\mu(h)=\mu_\infty(h)$.

\smallskip

To conclude the proof of Proposition~\ref{prop:mu}, note that
Lemma~\ref{lem:mu:bounds} implies for every $h\in\R$ the bounds
$\min\{f(h)/3,f(h)^2/(9\bar\gamma)\} \leq \mu(h)\leq f(h)$, so it only
remains to show the Lipschitz continuity of the function $\mu$ that
maps $h$ to $\mu(h)$. This is straightforward and relies on the same
argument as in~\cite{GZ25concentration} for i.i.d.\ environment.
Indeed, a direct computation of $\partial_h \log
\bbE[1/Z_{n,h}^{\textbf{-}}]$, which we omit for the sake of brevity,
yields $|\mu(h) - \mu(h')| \leq |h-h'|$ for all $h,h'\in\R$.
\end{proof}

\section{Decay of correlations in the localized phase}\label{sec:quenchedcorrelations}
The goal of this section is to investigate the correlations between
polymer contact points under the quenched Gibbs measure
$\prob_{n,h,\go}$.  

\subsection{Decorrelation lemma}
We first state a decorrelation lemma that will be used throughout the
section. It follows from Proposition~\ref{prop:mu} and the comparison
argument from Lemma~\ref{lem:comparison}. 

\begin{lemma}
  \label{lem:decorrelation}
For any $h>h_c$ and $\varkappa\in(0,1)$ there exist $C>0$,
$\lambda>0$, and $\gd>0$ such that for all $n\in\N$ and integers
$0\leq a_1<b_1\leq a_2<\cdots <b_l\le n$ that satisfy $l\le \delta n$
and $\sum_{i=1}^l (b_i-a_i) \geq \varkappa n$ 
\begin{equation*}
\bbE\bigg[ \prod_{i=1}^l \prob_{b_i-a_i,h,\vartheta^{a_i}\cdot} [T_1=b_i-a_i] \bigg] \leq C^l \ee^{-\lambda n}\,.
\end{equation*}
\end{lemma}

\begin{proof}[Proof of Lemma~\ref{lem:decorrelation}]
Fix $h>h_c$ and $\varkappa\in(0,1)$.  As in the proof of
Proposition~\ref{prop:mu}, for $r\in\N$ introduce the covariance
function $\gamma_r$ defined by
$\gamma_r(0)\ceq\gamma(0)+2\sum_{c=r+1}^\infty|\gamma(c)|$ and
$\gamma_r(c)\ceq\ind_{\{c\le r\}}\gamma(c)$ for $c\in\N$, where
$\gamma$ is the covariance function associated with $\probd$. Let
$\probd_r$ be the law on $(\Omega,\mathcal{F})$ with covariance
function $\gamma_r$ that satisfies Assumption~\ref{assum:gauss} and
denote its expectation by $\Exd_r$. Finally, consider the alternative
free energy $\mu_r(h)\ceq\lim_{n\uparrow\infty}-(1/n)\log
\bbE_r[1/Z_{n,h}^{\textbf{-}}]$ in the environment with law
$\probd_r$, which exists by Proposition~\ref{prop:mu}. In the proof of
Proposition~\ref{prop:mu} we demonstrated that
$\mu(h)=\mu_\infty(h)\ceq\lim_{r\uparrow\infty}\mu_r(h)$. Then, as
$\mu(h)>0$ because $h>h_c$, we can take an $r$ so large that
$\sum_{c=r+1}^\infty|\gamma(c)|\leq \mu(h)\varkappa/192$ and
$\mu_r(h)\geq \mu(h)/2$. The latter implies that there exists $C\ge 2$
such that $\bbE_r[1/Z_{n,h}^{\textbf{-}}]\le (C/2)\ee^{-\mu(h)n/4}$
for all $n\in\N$.  We claim that the lemma is verified with the
constants $C$, $\lambda\ceq \mu(h)\varkappa/32$, and
$\gd\ceq\varkappa/2r$.

Pick $n,\ell\in\N$ and integers $0\leq a_1<b_1\leq a_2<\cdots <b_l\le n$
that satisfy $l\le \delta n$ and $\sum_{i=1}^l (b_i-a_i) \geq
\varkappa n$.  Lemma~\ref{lem:comparison} yields
\begin{equation*}
  \Exd \bigg[\prod_{i=1}^l\frac {1}{Z_{b_i-a_i,h}^{\textbf{-}}(\vartheta^{a_i}\cdot)}\bigg]  \leq
  \Exd_r\bigg[\prod_{i=1}^l\frac {1}{Z_{b_i-a_i,h}^{\textbf{-}}(\vartheta^{a_i}\cdot)}\bigg] \ee^{6n\sum_{c=r+1}^\infty|\gamma(c)|} \,.
\end{equation*}
Recalling that $\prob_{t,h,\cdot}[T_1=t] = p(t)/Z^{\textbf{-}}_{t,h}$
for all $t\in\N$ and using
$\sum_{c=r+1}^\infty|\gamma(c)|\leq \mu(h)\varkappa/192$, we then obtain
\begin{equation}
  \Exd \bigg[\prod_{i=1}^l\prob_{b_i-a_i,h,\gth^{a_i}\cdot} [T_1=b_i-a_i]\bigg]
  \leq\Exd_r\bigg[\prod_{i=1}^l\prob_{b_i-a_i,h,\gth^{a_i}\cdot} [T_1=b_i-a_i]\bigg] \ee^{\frac{\mu(h)\varkappa}{32} n}\,.
\label{eq:lem:decorrelation:main}
\end{equation}
We now restrict the product in the right-hand side of
\eqref{eq:lem:decorrelation:main} to some of those indices $i$ for
which $b_i-a_i> r$, using the bound $\prob_{b_i-a_i,h,\gth^{a_i}\cdot}
[T_1=b_i-a_i]\le 1$ to discard the terms corresponding to the
remaining indices $i$.  To this aim, let $i_1<i_2<\cdots<i_m$ with
$m\le l$ denote all the indices $i\in\{1,\ldots,l\}$ such that
$b_i-a_i>r$. We observe that
\begin{equation*}
\sum_{j=1}^m (b_{i_j}-a_{i_j}) = \sum_{i=1}^l (b_i-a_i) - \sum_{i=1}^l (b_i-a_i) \ind_{\{b_i-a_i \leq r\}} \geq \varkappa n-rl\ge(\varkappa-\gd r) n=\frac{\varkappa}{2}n\,,
\end{equation*}
which shows in particular that $m>0$.  Splitting the sum $\sum_{j=1}^m
(b_{i_j}-a_{i_j})$ into odd and even labels, we also see that either
\begin{equation*}
    \sum_{j\in J_o} (b_{i_j}-a_{i_j}) \geq \frac{\varkappa}{4}n \quad\text{or}\quad \sum_{j\in J_e} (b_{i_j}-a_{i_j})  \geq \frac{\varkappa}{4}n\,,
\end{equation*}
where $J_o\ceq\{j\in\{1,\ldots,m\}: j \text{ is odd}\}$ and
$J_e\ceq\{j\in\{1,\ldots,m\}: j \text{ is even}\}$. Then
\eqref{eq:lem:decorrelation:main} implies
\begin{equation}
  \Exd \bigg[\prod_{i=1}^l\prob_{b_i-a_i,h,\gth^{a_i}\cdot} [T_1=b_i-a_i]\bigg]
  \leq\Exd_r\bigg[\prod_{j=1}^m\prob_{b_{i_j}-a_{i_j},h,\gth^{a_{i_j}}\cdot} [T_1=b_{i_j}-a_{i_j}]\bigg] \ee^{\frac{\mu(h)\varkappa}{32} n}\\
\label{eq:lem:decorrelation:0}
\end{equation}
with
\begin{align}
  \nonumber
  &\Exd_r\bigg[\prod_{j=1}^m\prob_{b_{i_j}-a_{i_j},h,\gth^{a_{i_j}}\cdot} [T_1=b_{i_j}-a_{i_j}]\bigg]\\
  \nonumber
&\qquad\leq \ind_{\big\{\sum_{j\in J_o} (b_{i_j}-a_{i_j})\geq \frac{\varkappa}{4}n\big\}} \Exd_r\bigg[\prod_{j\in J_o}\prob_{b_{i_j}-a_{i_j},h,\gth^{a_{i_j}}\cdot} [T_1=b_{i_j}-a_{i_j}]\bigg]\\
  &\qquad\quad +  \ind_{\big\{\sum_{j\in J_e} (b_{i_j}-a_{i_j})\geq \frac{\varkappa}{4}n\big\}} \Exd_r\bigg[\prod_{j\in J_e}\prob_{b_{i_j}-a_{i_j},h,\gth^{a_{i_j}}\cdot} [T_1=b_{i_j}-a_{i_j}]\bigg]\,.
  \label{eq:lem:decorrelation:1}
\end{align}
Since $\gamma_r(c)=0$ for all $c> r$, we note that the factors
$\prob_{b_{i_j}-a_{i_j},h,\gth^{a_{i_j}}\cdot} [T_1=b_{i_j}-a_{i_j}]$
with $j\in J_o$ are independent under $\bbP_r$. The same is true for
the factors $\prob_{b_{i_j}-a_{i_j},h,\gth^{a_{i_j}}\cdot}
[T_1=b_{i_j}-a_{i_j}]$ with $j\in J_e$.  Moreover, the constant $C$
that we have defined above allows us to write for all $j$
\begin{equation*}
\Exd_r\Big[\prob_{b_{i_j}-a_{i_j},h,\gth^{a_{i_j}}\cdot} [T_1=b_{i_j}-a_{i_j}]\Big]\le \Exd_r\bigg[\frac{1}{Z^{\textbf{-}}_{b_{i_j}-a_{i_j},h}}\bigg]\le \frac{C}{2}\ee^{-\frac{\mu(h)}{4}(b_{i_j}-a_{i_j})}\,.
\end{equation*}
In this way, \eqref{eq:lem:decorrelation:1} becomes
\begin{align*}
&\Exd_r\bigg[\prod_{j=1}^m\prob_{b_{i_j}-a_{i_j},h,\gth^{a_{i_j}}\cdot} [T_1=b_{i_j}-a_{i_j}]\bigg]\\
&\qquad\leq \ind_{\big\{\sum_{j\in J_o} (b_{i_j}-a_{i_j})\geq \frac{\varkappa}{4}n\big\}} \prod_{j\in J_o}\Exd_r\Big[\prob_{b_{i_j}-a_{i_j},h,\gth^{a_{i_j}}\cdot} [T_1=b_{i_j}-a_{i_j}]\Big]\\
&\qquad\quad +  \ind_{\big\{\sum_{j\in J_e} (b_{i_j}-a_{i_j})\geq \frac{\varkappa}{4}n\big\}} \prod_{j\in J_e}\Exd_r\Big[\prob_{b_{i_j}-a_{i_j},h,\gth^{a_{i_j}}\cdot} [T_1=b_{i_j}-a_{i_j}]\Big]\\
     &\qquad \leq 2 \bigg(\frac{C}{2}\bigg)^{\lceil \frac{m}{2}\rceil} \ee^{-\frac{\mu(h)\varkappa}{16}n}
     \leq C^l \ee^{-\frac{\mu(h)\varkappa}{16}n}\,,
\end{align*}
and \eqref{eq:lem:decorrelation:0} gives
\begin{equation*}
  \Exd \bigg[\prod_{i=1}^l\prob_{b_i-a_i,h,\gth^{a_i}\cdot} [T_1=b_i-a_i]\bigg]
  \leq C^l \ee^{-\frac{\mu(h)\varkappa}{16}n+\frac{\mu(h)\varkappa}{32} n}=C^l\ee^{-\lambda n}\,.
\qedhere
\end{equation*}
\end{proof}

\begin{remark}
{\rm Let us mention that Lemma~\ref{lem:decorrelation} could also be
  obtained by a direct application of the hypercontractivity
  inequality from Proposition~\ref{prop:hypercontract}, under the additional
  hypothesis that there exists $\eta>0$ such that
  $\sum_{a=1}^n\sum_{b=1}^nu_a\gamma(|a-b|)u_b\ge\eta$ for all
  $n\in\N$ and $(u_1,\ldots,u_n)\in\mathbb{S}_n$. This would make for
  a shorter, more elegant proof; however, we stuck with the longer
  proof in order to cover the most general case.}
\end{remark}

\subsection{Polymer correlations in the localized phase}
In order to investigate the correlations under the Gibbs measure in
the localized phase, we require the following two lemmas: the first
one provides a lower bound on the number of contact points, while the
second one states that large excursions $T_i$ with $1\le i\le L_n$ are
unlikely to cover a macroscopic subset of $\{1,\ldots,n\}$.  They were
proved for the i.i.d.\ case in~\cite[Lemma~2.3]{GZ25concentration}
and~\cite[Lemma~3.3]{GZ25concentration}, respectively. In this paper
we adapt them to a Gaussian field with summable correlations by
replacing the independence assumption with
Lemma~\ref{lem:decorrelation}.

\begin{lemma}\label{lem:supercrit:contactLB}
For every closed set $H\subset(h_c,+\infty)$ there exist $\gd>0$
and $\lambda>0$ such that for all $n\in\N$
\begin{equation*}
\bbE\bigg[\sup_{h\in H} \prob_{n,h,\cdot}[L_n < \gd n]\bigg] \leq \ee^{-\lambda n}\,.
\end{equation*}
\end{lemma}

Measurability of $\sup_{h\in H} \prob_{n,h,\cdot}[L_n < \gd n]$ is not
an issue since $h\mapsto \prob_{n,h,\go}[L_n < \gd n]$ is continuous
for all $\omega\in\Omega$.

\begin{proof}[Proof of Lemma~\ref{lem:supercrit:contactLB}]
Fix a closed subset $H\subset(h_c,+\infty)$ and put $h_o\ceq \inf H >
h_c$. Recalling the quenched computations
of~\cite[Lemma~2.3]{GZ25concentration}, one obtains for all $n\in\N$,
$\go\in\gO$, and $\delta\in(0,1]$
\begin{equation}
  \sup_{h\in H} \prob_{n,h,\go}[L_n<\gd n] \leq \sum_{l\in\N}\ind_{\{l<\gd n\}}\sum_{0=b_0<b_1<\cdots<b_l=n} \,\prod_{i=1}^l \prob_{b_i-b_{i-1},h_o,\gth^{b_{i-1}}\go} [T_1=b_i-b_{i-1}]\,.
  \label{lem:supercrit:contactLB_0}
\end{equation}
At the same time, applying Lemma~\ref{lem:decorrelation} with sites
$a_i=b_{i-1}$ and, e.g., $\varkappa=1/2$, we deduce that there exist
positive constants $C_o$, $\lambda_o$, and $\delta_o$ (depending on
$h_o>h_c$) such that for all $n\in\N$ and $0=b_0<b_1<\cdots<b_l=n$
with $l\le\delta_on$
\begin{equation}
  \label{eq:decorrelationinequality1}
 \Exd\bigg[\prod_{i=1}^l \prob_{b_i-b_{i-1},h_o,\gth^{b_{i-1}}\cdot} [T_1=b_i-b_{i-1}]\bigg]\leq C_o^l \ee^{-\lambda_o n}\,.
\end{equation}
Therefore, integrating \eqref{lem:supercrit:contactLB_0} with respect
to $\probd[\dd\omega]$ and then appealing to
\eqref{eq:decorrelationinequality1}, one finds for all $n\in\N$ and
positive $\delta\le \min\{1,\delta_o\}$
\begin{equation*}
\bbE\left[\sup_{h\in H} \prob_{n,h,\cdot}[L_n < \gd n]\right]\leq \sum_{l\in\N}\ind_{\{l<\gd n\}} \binom{n-1}{l-1} C_o^l \ee^{-\lambda_o n}\,.
\end{equation*}
Set $\lambda\ceq\lambda_o/3$ and $\gd\ceq
\min\{1,\delta_o,\lambda/\log(1+ C_o/\lambda)\}$. Then, using
$\ind_{\{l<\gd n\}}\leq (1+C_o/\lambda)^{\gd n-l}\leq C_o^{-l}
\lambda^l \ee^{\lambda n}$ for $l\le n$, one finally get
\begin{align*}
  \bbE\left[\sup_{h\in H} \prob_{n,h,\cdot}[L_n < \gd n]\right]\leq \sum_{l=1}^n\binom{n-1}{l-1}\lambda^l\ee^{-2\lambda n}
= \lambda(1+\lambda)^{n-1}\ee^{-2\lambda n}\leq \ee^{-\lambda n}\,,
\end{align*}
which concludes the proof.
\end{proof}

\begin{lemma}
  \label{lem:supercrit:finitejumps}
For any closed set $H\subset(h_c,+\infty)$ and $\varkappa\in(0,1)$ there
exist $\lambda>0$ and $\tau\in\N$ such that for all $n\in\N$
\begin{equation}
\bbE\Bigg[\sup_{h\in H} \prob_{n,h,\cdot}\bigg[ \sum_{i=1}^{L_n}T_i\ind_{\{T_i>\tau\}} \ge \varkappa n\bigg]\Bigg] \leq \ee^{-\lambda n}\,.
\end{equation}
\end{lemma}

\begin{proof}[Proof of Lemma~\ref{lem:supercrit:finitejumps}]
The proof is quite similar to that of
Lemma~\ref{lem:supercrit:contactLB}. Fix a closed set
$H\subset(h_c,+\infty)$ and put $h_o\ceq \inf H > h_c$. Pick
$\varkappa\in(0,1)$. Note that if there exist $l\ge1$ indices $1\le
i_1<\cdots<i_l\le L_n$ such that $T_{i_j}>\tau$ for all
$j\in\{1,\ldots,l\}$, then $l<n/\tau$. Therefore, the same quenched
computations as in~\cite[Lemma~3.3]{GZ25concentration} yield that for all
$n\in\N$, $\omega\in\Omega$, and $\tau\in\N$
\begin{align}
\notag
&\sup_{h\in H} \prob_{n,h,\go}\bigg[ \sum_{i=1}^{L_n}T_i\ind_{\{T_i>\tau\}} \ge \varkappa  n\bigg] 
\\\label{eq:lem:supercrit:finitejumps:1}
&\quad\leq \sum_{l\in\N}\ind_{\{l<\frac n\tau\}}\sum_{0\le a_1<b_1\le\cdots<b_l\le n} \ind_{\{\sum_{i=1}^l(b_i-a_i)\ge\varkappa n\}} \prod_{i=1}^l  \prob_{b_i-a_i,h_o,\gth^{a_i}\go}[T_1=b_i-a_i]\,.
\end{align}
Recalling Lemma~\ref{lem:decorrelation}, there exist three positive
constants $C_o$, $\lambda_o$, and $\gd_o$ (depending on $h_o$ and
$\varkappa$) such that for every integers $0\le a_1<b_1\le\cdots<b_l\leq
n$ with $l\le \delta_o n$ and $\sum_{i=1}^l(b_i-a_i)\ge\varkappa n$
\begin{equation*}
\bbE\left[ \prod_{i=1}^l \prob_{b_i-a_i,h_o,\gth^{a_i}\cdot} [T_1=b_i-a_i] \right] \leq C_o^l \ee^{-\lambda_o n}\,.
\end{equation*}
Then, integrating with respect to $\bbP[\dd \go]$
in~\eqref{eq:lem:supercrit:finitejumps:1}, we eventually obtain for
all $n,\tau\in\N$ with $\tau\geq 1/\delta_o$
\begin{equation*}
\bbE\Bigg[\sup_{h\in H} \prob_{n,h,\cdot}\bigg[ \sum_{i=1}^{L_n}T_i\ind_{\{T_i>\tau\}} \ge \varkappa n\bigg]\Bigg]
\leq \sum_{l\in\N}\ind_{\{l<\frac n\tau\}} \binom{n+l}{2l} C_o^l \ee^{-\lambda_o n}\,.
\end{equation*}
Following from this bound, the proof of the lemma is concluded by a
direct analysis that gives appropriate values $\lambda>0$ and
$\tau\in\N$ in terms of $C_o$, $\lambda_o$, and $\delta_o$. Since this
argument is a carbon copy of the analysis
in~\cite[Lemma~3.3]{GZ25concentration}, we omit the details here for
the sake of brevity.
\end{proof}

We finally introduce the main result of this section, which is the
exponential decay of correlations between contact points of the
polymer under the Gibbs measure in the localized phase. As a matter of
fact, we present a stronger statement on the decorrelation between two
copies of the polymer, which implies the aforementioned result. These
statements play a key role in the remainder of the analysis of the
pinning model in its localized phase.

\begin{lemma}
  \label{lem:supercrit:polymercorrelations}
Let $(S_i)_{i\in\N_0}$ and $(S'_i)_{i\in\N_0}$ be two independent
copies of the renewal process and let $\Ex^{\otimes2}_{n,h,\omega}$
denote the expectation under the product law of their quenched Gibbs
measures in the same environment $\go\in\gO$.  Then for any closed set
$H\subset(h_c,+\infty)$ there exist $C>0$ and $\lambda>0$ such that
for all $n\in\N$
\begin{equation*}
\bbE\Bigg[\sup_{h\in H} \Ex^{\otimes2}_{n,h,\cdot}\bigg[\prod_{k=1}^{n-1}(1-X_kX'_k)\bigg]\Bigg] \leq C \ee^{-\lambda n}\,.
\end{equation*}
\end{lemma}

\begin{proof}[Proof of Lemma~\ref{lem:supercrit:polymercorrelations}]
This result is analogous to~\cite[Lemma~3.2]{GZ25concentration} for
the i.i.d.\ case. One can check therein that the lemma is a
consequence of some quenched computations together with
Lemmas~\ref{lem:supercrit:contactLB} and
\ref{lem:supercrit:finitejumps} and the following claim: for all
$\gd>0$ there exists $z>0$ such that for all $n\in\N$
\[
\int_\Omega\ind_{\big\{\sum_{a=1}^n \ind_{\{\go_a<-z\}}>\gd n\big\}}\probd[\dd\omega] \leq \ee^{-\gd n}\,.
\]
In the i.i.d.\ case this is a direct consequence of a Chernoff-type
bound. We now prove that the claim also holds under
Assumption~\ref{assum:gauss}. This will complete the proof of the
lemma with a carbon copy of the arguments in~\cite{GZ25concentration},
which we do not reproduce here for the sake of brevity.

\smallskip

Since, for every positive $\gd$ and $z$, the condition
$\sum_{a=1}^n \ind_{\{\go_a<-z\}}>\gd n$ implies that there exists
$A\subseteq\{1,\ldots,n\}$ with $|A|>\gd n$ such that $\omega_a<-z$
for all $a\in A$, which in turn entails $\sum_{a\in
  A}\omega_a<-z\delta n$, one has
\begin{equation*}
\int_\Omega\ind_{\big\{\sum_{a=1}^n \ind_{\{\go_a<-z\}}>\gd n\big\}}\probd[\dd\omega]  \leq \sum_{A\subseteq\{1,\ldots,n\}}\int_\Omega\ind_{\big\{\sum_{a\in A} \go_a<-z \gd n\big\}} \probd[\dd\omega]\,.
\end{equation*}
However, for any non-empty subset $A\subseteq\{1,\ldots,n\}$ the sum
$\sum_{a\in A}\go_a$ is a centered Gaussian random variable with
variance $\sum_{a\in A}\sum_{b\in A}\gamma(|a-b|)\leq 2\bar\gga n$.
In particular, a classical Gaussian tail estimate yields for all
$n\in\N$
\[
\int_\Omega\ind_{\big\{\sum_{a\in A} \go_a<-z \gd n\big\}} \probd[\dd\omega]\leq \ee^{-\frac{z^2\gd^2 }{4\bar\gga}n}\,,
\]
so
\begin{align*}
\int_\Omega\ind_{\big\{\sum_{a=1}^n \ind_{\{\go_a<-z\}}>\gd n\big\}}\probd[\dd\omega]\leq 2^n\ee^{-\frac{z^2\gd^2 }{4\bar\gga}n}\,.
\end{align*}
This demonstrates the claim taking $z\ceq \sqrt{4\bar\gga
  (\gd^{-1}+\gd^{-2}\log 2)}$.
\end{proof}

Let $\cov_{n,h,\go}$ denote the covariance under the polymer law
$\prob_{n,h,\go}$. Then we have the following corollary of
Lemma~\ref{lem:supercrit:polymercorrelations}.

\begin{corollary}
  \label{corol:correlations:polym}
For any closed set $H\subset(h_c,+\infty)$ there exist $C>0$ and
$\lambda>0$ such that for all integers $0\le a\le b\le n$
\begin{equation*}
\bbE\left[\sup_{h\in H} \frac{\left|\cov_{n,h,\cdot}[X_a,X_b]\right|}{\min\{\Ex_{n,h,\cdot}[X_a],\Ex_{n,h,\cdot}[X_b]\}}\right] \leq C \ee^{-\lambda (b-a)}\,.
\end{equation*}
\end{corollary}

\begin{proof}[Proof of Corollary~\ref{corol:correlations:polym}]
This statement is analogous to Lemma~3.4 in~\cite{GZ25concentration}
for the case of i.i.d.\ disorder.  It is a consequence of some
quenched computations together with
Lemma~\ref{lem:supercrit:polymercorrelations} and a Chernoff-type
bound on the charges.  In particular, these arguments also apply in the
case of the environment defined by Assumption~\ref{assum:gauss}, so we
do not reproduce the proof here.
\end{proof}

\section{Largest excursion, regularity and quenched CLT}
\label{sec:proofmainresults}

In this section we prove our three main results on the quenched
analysis of the localized phase of the pinning polymer model in an
environment satisfying Assumption~\ref{assum:gauss}. They concern the
size of the largest gap between contact points
(Proposition~\ref{prop:largestjump}), the smoothness of the free
energy (Theorem~\ref{thm:fsmooth}), and the quenched CLT for the total
number of contact points (Theorem~\ref{thm:CLT:Ln:quenched}), together
with a corresponding concentration inequality. These findings
constitute direct expansions of the quenched results established
in~\cite{GZ25concentration} from the i.i.d.\ setting to the case of a
Gaussian field with summable correlations, and follow primarily from
Lemma~\ref{lem:supercrit:polymercorrelations}. Accordingly, we focus
here on the elements of the proofs that differ
from~\cite{GZ25concentration} -- most notably the proof
of~\eqref{eq:largestjump:LB} below -- while omitting those arguments
that carry over unchanged, such as the quenched computations.

\subsection{Largest excursion}
The proof of Proposition~\ref{prop:largestjump} comes from the
following bounds: for all $h_o>h_c$ and $\eps\in(0,1]$ and for $\pae$
\begin{equation}
  \label{eq:largestjump:UB}
\adjustlimits\lim_{n\uparrow\infty}\sup_{h\in U_{h_o,\eps}} \prob_{n,h,\go}\left[\frac{M_n}{\log n}>\frac{1+\eps}{\mu(h)}\right]=0
\end{equation}
and
\begin{equation}
  \label{eq:largestjump:LB}
\adjustlimits\lim_{n\uparrow\infty}\sup_{h\in U_{h_o,\eps}} \prob_{n,h,\go}\left[\frac{M_n}{\log n}<\frac{1-\eps}{\mu(h)}\right]=0\,,
\end{equation}
where $U_{h_o,\eps}$ denotes the open interval centered at $h_o$ with
radius $\gd_o\ceq \min\{\eps\mu(h_o)/16, (h_o-h_c)/2\}$. Indeed,
Proposition~\ref{prop:largestjump} then follows from the same
compactness argument as in the proof of~\cite[Proposition
  1.4]{GZ25concentration}, which is not reproduced here for the sake
of conciseness. Note that the present value of $\delta_o$ is smaller
than the one used in \cite{GZ25concentration}, which is
$\min\{\eps\mu(h_o)/8, (h_o-h_c)/2\}$: this is because verifying
\eqref{eq:largestjump:LB} with correlated Gaussian disorder requires
more involved technical estimates than those needed in the case of
i.i.d.\ charges.

\begin{proof}[Proof of~\eqref{eq:largestjump:UB}]
Given $h_o>h_c$ and $\eps\in(0,1]$, one deduces from some quenched
computations (see~\cite[Equation~(37)]{GZ25concentration}) that for all
$n\in\N$ and $\go\in\gO$
\[
\sup_{h\in U_{h_o,\eps}} \prob_{n,h,\go}\left[\frac{M_n}{\log n}>\frac{1+\eps}{\mu(h)}\right] \leq n^{-\gz_o\frac{1+\eps}{\mu_o+\gd_o}} \sum_{i=0}^{n-1} \gL(\gth^i\go)\,,
\]
where $\mu_o\ceq \mu(h_o)>0$, $\gz_o\ceq (1-\eps/4)\mu_o$, $\gth$ is
the shift operator on $\Omega$, and $\gL$ is the following random
variable on $\Omega$:
\[
\gL\ceq \sum_{j\in\N} \ee^{(\gz_o+\gd_o)j}\prob_{j,h_o,\cdot}[T_1=j]\le\sum_{j\in\N} \frac{\ee^{(\gz_o+\gd_o)j}}{Z_{j,h_o}^{\textbf{-}}}\,.
\]
Proposition~\ref{prop:mu} yields $\bbE[\gL]<+\infty$ since
$\zeta_o+\delta_o<\mu_o$. On the other hand, a direct analysis shows
that $\gz_o\frac{1+\eps}{\mu_o+\gd_o}\geq 1+\eps/3$. In this way,
recalling that $\vartheta$ is an ergodic transformation within our
setting, we obtain~\eqref{eq:largestjump:UB} by applying Birkhoff's
ergodic theorem.
\end{proof}

\begin{proof}[Proof of~\eqref{eq:largestjump:LB}]

The proof of the lower bound~\eqref{eq:largestjump:LB} is more
involved and additionally requires Gaussian
hypercontractivity. In~\cite{GZ25concentration} the inequality
\eqref{eq:largestjump:LB} is obtained for i.i.d.\ disorder through a
\textit{dilution procedure} that considers only trajectories making
exactly $r_n$ jumps of length $m_n$ at a distance $d_n$ from one
another, where $r_n$, $m_n$, and $d_n$ are suitably chosen. The
following bound summarize the outcome of this dilution procedure.
Given $h_o>h_c$ and $\eps\in(0,1]$, we define $\mu_o\ceq\mu(h_o)>0$
for brevity and we set
\[
m_n\ceq \left\lfloor\tfrac{1-\eps}{\mu_o-\gd_o}\log n\right\rfloor +1\,,\quad d_n \ceq \left\lfloor (\log n)^2 \right\rfloor +1\,,\quad
\text{and}\quad r_n \ceq \left\lfloor \tfrac{n}{m_n+d_n}\right\rfloor\,.
\]
Then, some quenched computations
(see~\cite[Equation~(39)]{GZ25concentration}) yield for all $n\in\N$ and
$\omega\in\Omega$ the bound
\begin{align}
  \nonumber
  &\sup_{h\in U_{h_o,\eps}}\prob_{n,h,\omega}\bigg[\frac{M_n}{\log n}<\frac{1-\epsilon}{\mu(h)}\bigg]\\
  \nonumber
  &\qquad\le \prod_{s=0}^{r_n-1}\Bigg\{1-\inf_{h\in U_{h_o,\eps}}\Ex_{n,h,\omega}\bigg[X_{l_s}\prod_{k=l_s+1}^{l_s+m_n-1}(1-X_k)X_{l_s+m_n}\bigg]\Bigg\}\\
  &\qquad\quad+2\sum_{s=0}^{r_n-2}\,\sum_{i=0}^{l_s+m_n-1}\sum_{j=l_{s+1}}^\infty\sup_{h\in [h_o-\delta_o,+\infty)}\Ex_{j-i,h,\vartheta^i\omega}^{\otimes 2}\bigg[\prod_{k=1}^{j-i-1}(1-X_kX_k')\bigg]\,,
    \label{eq:largestjump:LB_0}
\end{align}
where $l_s\ceq s(m_n+d_n)$. 
\smallskip

Let us first show that for $\pae$
\begin{equation}
  \lim_{n\uparrow\infty}
  \sum_{s=0}^{r_n-2}\,\sum_{i=0}^{l_s+m_n-1}\sum_{j=l_{s+1}}^\infty \sup_{h\in [h_o-\delta_o,+\infty)}\Ex_{j-i,h,\vartheta^i\omega}^{\otimes 2}\bigg[\prod_{k=1}^{j-i-1}(1-X_kX_k')\bigg]=0\,.
    \label{eq:largestjump:LB_1}
\end{equation}
Indeed, as $h_o-\delta_o>h_c$, according to Lemma
\ref{lem:supercrit:polymercorrelations} there exist positive constants
$C$ and $\lambda$ such that
\begin{equation*}
\Exd\Bigg[\sup_{h\in[h_o-\delta_o,+\infty)}\Ex_{j,h,\cdot}^{\otimes 2}\bigg[\prod_{k=1}^{j-1}(1-X_kX_k')\bigg]\Bigg]\le C\ee^{-\lambda j}
\end{equation*}
for all $j\in\N$. In this way, introducing the random variable
\begin{equation*}
  \Lambda:=\sum_{j=1}^\infty \ee^{\frac{\lambda}{2}j}\sup_{h\in[h_o-\delta_o,+\infty)}\Ex_{j,h,\cdot}^{\otimes 2}\bigg[\prod_{k=1}^{j-1}(1-X_kX_k')\bigg]\,,
\end{equation*}
we have $\Exd[\Lambda]<+\infty$ and a Chernoff-type bound gives for
all $n\in\N$ and $\omega\in\Omega$
\begin{align}
  \nonumber
  &\sum_{s=0}^{r_n-2}\,\sum_{i=0}^{l_s+m_n-1}\sum_{j=l_{s+1}}^\infty\sup_{h\in[h_o-\delta_o,+\infty)}\Ex_{j-i,h,\vartheta^i\omega}^{\otimes 2}\bigg[\prod_{k=1}^{j-i-1}(1-X_kX_k')\bigg]\\
  \nonumber
  &\qquad\le
  r_n\sum_{i=0}^{n-1}\sum_{j=d_n}^\infty \sup_{h\in[h_o-\delta_o,+\infty)}\Ex_{j,h,\vartheta^i\omega}^{\otimes 2}\bigg[\prod_{k=1}^{j-1}(1-X_kX_k')\bigg]
    \le  r_n\ee^{-\frac{\lambda}{2}d_n}\sum_{i=0}^{n-1}\Lambda(\vartheta^i\omega)\,.
\end{align}
Thus, \eqref{eq:largestjump:LB_1} follows from Birkhoff's ergodic theorem
since $r_n\le n$ and $d_n\ge (\log n)^2$.

\smallskip

In the light of \eqref{eq:largestjump:LB_0} and
\eqref{eq:largestjump:LB_1}, \eqref{eq:largestjump:LB} follows if we
demonstrate that for $\pae$,
\begin{equation*}
  \lim_{n\uparrow\infty}
  \prod_{s=0}^{r_n-1}\Bigg\{1-\inf_{h\in U_{h_o,\eps}}\Ex_{n,h,\omega}\bigg[X_{l_s}\prod_{k=l_s+1}^{l_s+m_n-1}(1-X_k)X_{l_s+m_n}\bigg]\Bigg\}=0\,.
\end{equation*}
To this aim, recall the constant $\xi$ defined
in \eqref{eq:defxi}. Then some more quenched computations (see Equations~(42--44) in
\cite{GZ25concentration} before that the restrictions $\omega_{l_s}\ge
0$ and $\omega_{l_s+m_n}\ge 0$ are imposed) yield that
\begin{align}
  \nonumber
  &\inf_{h\in U_{h_o,\eps}}\Ex_{n,h,\omega}\bigg[X_{l_s}\prod_{k=l_s+1}^{l_s+m_n-1}(1-X_k)X_{l_s+m_n}\bigg]\\
  \nonumber
  &\qquad\ge\frac{1-\sup_{h\in U_{h_o,\eps}}\prob_{n,h,\omega}\big[\frac{M_n}{\log n}>
      \frac{1+\epsilon}{\mu(h)}\big]}{1+\xi\ee^{-h_o+\delta_o-\omega_{l_s}}\big(\frac{1+\epsilon}{\mu_o-\delta_o}\log n\big)^\xi}\,
  \frac{1}{1+\xi \ee^{-h_o+\delta_o-\omega_{l_s+m_n}}m_n^\xi}\\
   &\qquad\qquad\qquad\qquad\qquad\qquad\qquad\qquad\qquad\quad\times\ee^{-\delta_o m_n}\prob_{m_n,h_o,\vartheta^{l_s}\omega}[T_1=m_n]\,.
 \label{eq:largestjump:LB_00}
\end{align}
Then, we remark that \eqref{eq:largestjump:UB} implies that for
$\pae$ the upper bound
\begin{equation*}
  \sup_{h\in U_{h_o,\epsilon}}\prob_{n,h,\omega}\bigg[\frac{M_n}{\log n}>\frac{1+\epsilon}{\mu(h)}\bigg] \le  \frac{1}{2}
\end{equation*}
is valid for all sufficiently large $n$. Moreover, as in the proof of
Lemma~\ref{lem:mu:convergencesupercrit}, a union bound and a standard
Gaussian tail estimate assure us that for all $n\in\N$
\begin{equation*}
  \int_\Omega\ind_{\big\{\max\{|\omega_0|,\ldots,|\omega_{n-1}|\}>\sqrt{6\gamma(0) \log n}\big\}}\probd[\dd\omega]\le \frac{2}{n^2}\,,
  \end{equation*}
so the Borel-Cantelli lemma implies that
$\max\{|\omega_0|,\ldots,|\omega_{n-1}|\}\le\sqrt{6\gamma(0) \log n}$
for all $n$ large enough for $\pae$. Combining
\eqref{eq:largestjump:LB_00} with these arguments we find for $\pae$
\begin{align}
  \nonumber
  &\limsup_{n\uparrow\infty}\prod_{s=0}^{r_n-1}\Bigg\{1-\inf_{h\in U_{h_o,\eps}}\Ex_{n,h,\omega}\bigg[X_{l_s}\prod_{k=l_s+1}^{l_s+m_n-1}(1-X_k)X_{l_s+m_n}\bigg]\Bigg\}\\
  &\qquad\qquad\qquad\qquad\qquad\qquad\le\limsup_{n\uparrow\infty}\prod_{s=0}^{r_n-1}\Big\{1-\zeta_n\ee^{-\delta_o m_n}\prob_{m_n,h_o,\vartheta^{l_s}\omega}[T_1=m_n]\Big\}\,,
   \label{eq:largestjump:LB_2}
\end{align}
where $\zeta_n$ is a numerical factor defined by
\begin{equation*}
  \zeta_n\ceq\frac{1}{2}\,\frac{1}{1+\xi\ee^{-h_o+\delta_o+\sqrt{6\gamma(0) \log n}}\big(\frac{1+\epsilon}{\mu_o-\delta_o}\log n\big)^\xi}\,
  \frac{1}{1+\xi \ee^{-h_o+\delta_o+\sqrt{6\gamma(0) \log n}}\,m_n^\xi}\,.
\end{equation*}
We shall prove that there exists $\lambda>0$ such that for
all $n$ large enough
\begin{equation}
  E_n\ceq \bbE\left[\prod_{s=0}^{r_n-1}\Big\{1-\zeta_n\ee^{-\delta_o m_n}\prob_{m_n,h_o,\vartheta^{l_s}\cdot}[T_1=m_n]\Big\}\right]\leq \exp\left(-\lambda\frac{\gz_nr_n}{n^{1-\eps/2}}\right)
  +\frac{1}{n^2}\,.
 \label{eq:largestjump:LB_3}
\end{equation}
This yields $\sum_{n\in\N}E_n<+\infty$ since
$\liminf_{n\uparrow\infty} (\log n)^{2\xi+2}\ee^{2\sqrt{6\gamma(0)
    \log n}}\gz_n r_n/n>0$, so a second application of the
Borel-Cantelli lemma entails that the right-hand side of
\eqref{eq:largestjump:LB_2} is zero for $\pae$, and this concludes the
proof of \eqref{eq:largestjump:LB}.

\smallskip

The bound \eqref{eq:largestjump:LB_3} is due to Gaussian
hypercontractivity. However we cannot directly apply Proposition
\ref{prop:hypercontract} to the expectation $E_n$ because we do not have any quantitative
assumption of the form $\sum_{a=1}^n\sum_{b=1}^nu_a\gamma(|a-b|)u_b\ge\eta$ for some
$\eta>0$ and all $n\in\N$, $(u_1,\ldots,u_n)\in\mathbb{S}_n$. Therefore we proceed in
two steps: 
we first perturb the random environment by adding small independent
Gaussian charges and after that we use Proposition
\ref{prop:hypercontract}. So, let $\probd'$ on $(\Omega,\mathcal{F})$
be the law of independent Gaussian random variables with mean zero and
variance $\eta>0$ so small that
\begin{equation}
  \eta\le\delta_o\quad\text{and}\quad\frac{\delta_o}{\sqrt{\eta}}\ge\frac{1}{2}+\log 2+3\frac{\mu_o-\delta_o}{1-\eps}\,.
\label{eq:eta_conditions}
\end{equation}
Let also $\probd_\ast$ be the convolution of the laws $\probd$ and
$\probd'$, which is a probability measure on $(\Omega,\mathcal{F})$
such that
$\int_\Omega\phi(\omega)\,\probd_\ast[\dd\omega]=\int_\Omega\probd'[\dd\omega']\int_\Omega\probd[\dd\omega]\phi(\omega+\omega')$
for all bounded measurable functions $\phi:\Omega\to\R$. Well-known
properties of Gaussian random variables imply that $\probd_\ast$
satisfies Assumption \ref{assum:gauss} with covariance function
$\gamma_\ast$ given by $\gamma_\ast(0)=\gamma(0)+\eta$ and
$\gamma_\ast(a)=\gamma(a)$ for $a\in\N$. Since
$\sum_{a=1}^n\sum_{b=1}^nu_a\gamma(|a-b|)u_b\ge0$ for all $n\in\N$ and
$(u_1,\ldots,u_n)\in\mathbb{S}_n$, we have
$\sum_{a=1}^n\sum_{b=1}^nu_a\gamma_\ast(|a-b|)u_b\ge\eta$ for all
$n\in\N$ and $(u_1,\ldots,u_n)\in\mathbb{S}_n$, so Proposition
\ref{prop:hypercontract} applies if we replace $\probd$ with
$\probd_\ast$.

We are now ready to demonstrate that \eqref{eq:largestjump:LB_3} holds
with $\lambda\ceq \eta/(2\eta+2\bar\gamma)$.  Using
$1\le\prod_{s=0}^{r_n-1}\ind_{\{\sum_{a=1}^{m_n}|\omega'_{l_s+a}|\le
  \delta_om_n\}}+\sum_{s=0}^{r_n-1}\ind_{\{\sum_{a=1}^{m_n}|\omega'_{l_s+a}|>\delta_om_n\}}$
for any $\omega'=(\omega'_a)_{a\in\N_0}\in\Omega$, we can write down
for all $n\in\N$ the bound
\begin{align}
  \nonumber
    E_n&\le\int_\Omega\probd'[\dd\omega']\int_\Omega\probd[\dd\omega]
    \prod_{s=0}^{r_n-1}\ind_{\big\{\sum_{a=1}^{m_n}|\omega'_{l_s+a}|\le \delta_om_n\big\}}\Big\{1-\zeta_n\ee^{-\delta_o m_n}\prob_{m_n,h_o,\vartheta^{l_s}\omega}[T_1=m_n]\Big\}\\
    &\quad+\sum_{s=0}^{r_n-1}\int_\Omega \ind_{\big\{\sum_{a=1}^{m_n}|\omega'_{l_s+a}|> \delta_om_n\big\}}\probd'[\dd\omega']\,.
    \label{eq:largestjump:LB_6}
\end{align}
However for all $\omega=(\omega_a)_{a\in\N_0}\in\Omega$,
$\omega'=(\omega_a')_{a\in\N_0}\in\Omega$, and
$s\in\{0,\ldots,r_n-1\}$ we have
\begin{align}
  \nonumber
  \prob_{m_n,h_o,\vartheta^{l_s}\omega}[T_1=m_n]&=\frac{p(m_n)}{\Ex\big[\ee^{\sum_{a=1}^{m_n-1}(h+\omega_{l_s+a})X_a}X_{m_n}\big]}\\
  \nonumber
  &\ge\frac{p(m_n)}{\Ex\big[\ee^{\sum_{a=1}^{m_n-1}(h+\omega_{l_s+a}+\omega'_{l_s+a})X_a+\sum_{a=1}^{m_n}|\omega'_{l_s+a}|}X_{m_n}\big]}\\
  &=\ee^{-\sum_{a=1}^{m_n}|\omega'_{l_s+a}|} \,\prob_{m_n,h_o,\vartheta^{l_s}(\omega+\omega')}[T_1=m_n]\,.
\label{eq:largestjump:LB_7}
\end{align}
At the same time, recalling that $\omega_0', \omega_1',\ldots$ are
independent normal charges with mean zero and variance $\eta$, a
Chernoff bound implies
\begin{align*}
   \sum_{s=0}^{r_n-1}\int_\Omega \ind_{\big\{\sum_{a=1}^{m_n}|\omega'_{l_s+a}|> \delta_om_n\big\}}\probd'[\dd\omega']&=r_n\int_\Omega \ind_{\big\{\sum_{a=1}^{m_n}|\omega'_a|> \delta_om_n\big\}}\probd'[\dd\omega']\\
   &\le r_n\bigg(\int_\Omega \ee^{\frac{|\omega'_0|-\delta_o}{\sqrt{\eta}}}\probd'[\dd\omega']\bigg)^{\!m_n} =
   r_n\bigg(2\ee^{\frac{1}{2}-\frac{\delta_o}{\sqrt{\eta}}}\bigg)^{m_n} \,.
\end{align*}
Since $r_n\le n$ and $m_n\ge [(1-\eps)/(\mu_o-\delta_o)]\log n$ by
definition, the second of the conditions \eqref{eq:eta_conditions}
satisfied by $\eta$ allows us to conclude that
\begin{equation}
   \sum_{s=0}^{r_n-1}\int_\Omega \ind_{\big\{\sum_{a=1}^{m_n}|\omega'_{l_s+a}|> \delta_om_n\big\}}\probd'[\dd\omega']\le\frac{1}{n^2}\,.
\label{eq:largestjump:LB_8}
\end{equation}
Combining \eqref{eq:largestjump:LB_7} with
$\sum_{a=1}^{m_n}|\omega'_{l_s+a}|>\delta_om_n$ and using
\eqref{eq:largestjump:LB_8}, bound \eqref{eq:largestjump:LB_6} then
shows that for all $n\in\N$
\begin{align*}
    E_n&\le\int_\Omega\probd'[\dd\omega']\int_\Omega\probd[\dd\omega]
    \prod_{s=0}^{r_n-1}\Big\{1-\zeta_n\ee^{-2\delta_om_n}\prob_{m_n,h_o,\vartheta^{l_s}(\omega+\omega')}[T_1=m_n]\Big\}+\frac{1}{n^2}\\
    &=\int_\Omega
    \prod_{s=0}^{r_n-1}\Big\{1-\zeta_n\ee^{-2\delta_om_n}\prob_{m_n,h_o,\vartheta^{l_s}\omega}[T_1=m_n]\Big\}\probd_\ast[\dd\omega]+\frac{1}{n^2}\,.
\end{align*}
In this way, invoking Proposition \ref{prop:hypercontract} and noting
that $(2\eta+2\bar\gamma)/\eta>1$, we find that
\begin{align*}
    E_n&\le\prod_{s=0}^{r_n-1}\Bigg\{1- \zeta_n\ee^{-2\delta_om_n}\int_\Omega
   \prob_{m_n,h_o,\vartheta^{l_s}\omega}[T_1=m_n]\,\probd_\ast[\dd\omega]\Bigg\}^{\!\frac{\eta}{2\eta+2\bar\gamma}}+\frac{1}{n^2}\\
    & =\Big\{1- \zeta_n\ee^{-2\delta_om_n}\Exd_\ast\big[\prob_{m_n,h_o,\cdot}[T_1=m_n]\big]\Big\}^{\!\lambda r_n}+\frac{1}{n^2}\,,
\end{align*}
where we have put $\lambda\ceq \eta/(2\eta+2\bar\gamma)$ into context
and we have denoted by $\Exd_\ast$ the expectation with respect to
$\probd_\ast$.  As a last manipulation, we appeal to Lemma
\eqref{lem:comparison} to restore the original law of the environment:
\begin{align*}
  \Exd_\ast\big[\prob_{m_n,h_o,\cdot}[T_1=m_n]\big]&=\Exd_\ast\bigg[\frac{p(m_n)}{Z_{m_n,h_o}^{\textbf{-}}}\bigg]\\
  &\ge\ee^{-\eta m_n}\Exd\bigg[\frac{p(m_n)}{Z_{m_n,h_o}^{\textbf{-}}}\bigg]=\ee^{-\eta m_n} \Exd\big[\prob_{m_n,h_o,\cdot}[T_1=m_n]\big]\,.
\end{align*}
Since the first of the conditions \eqref{eq:eta_conditions} is
$\eta\le\delta_o$, we conclude that for all $n\in\N$
\begin{equation}
    E_n\le\Big\{1- \zeta_n\ee^{-3\delta_om_n}\Exd\big[\prob_{m_n,h_o,\cdot}[T_1=m_n]\big]\Big\}^{\!\lambda r_n}+\frac{1}{n^2}\,.
\label{eq:largestjump:LB_11}
\end{equation}

We are almost done. As $p(t)$ decays polynomially with $t$,
Proposition \ref{prop:mu} assures us that for all sufficiently large
$n$
\begin{equation*}
    \Exd\big[\prob_{m_n,h_o,\cdot}[T_1=m_n]\big]=p(m_n)\Exd\bigg[\frac{1}{Z_{m_n,h_o}^{\mathbf{-}}}\bigg]\ge\ee^{-(\mu_o+\delta_o)(m_n-1)+3\delta_o}\,.
\end{equation*}
Using this bound in \eqref{eq:largestjump:LB_11} and recalling that
$m_n\ge [(1-\epsilon)/(\mu_o-\delta_o)]\log n$, for all sufficiently
large $n$ we get
\begin{align}
  \nonumber
  E_n&=\Big\{1- \zeta_n\ee^{-(\mu_o+4\delta_o)(m_n-1)}\Big\}^{\!\lambda r_n}+\frac{1}{n^2}\\
  \nonumber
  &\le \exp\bigg\{\!-\lambda\zeta_nr_n\ee^{-(\mu_o+4\delta_o)(m_n-1)}\bigg\}+\frac{1}{n^2}
  \le \exp\Bigg\{\!-\frac{\zeta_nr_n}{n^{(\mu_o+4\delta_o)\big(\frac{1-\epsilon}{\mu_o-\delta_o}\big)}}\Bigg\}+\frac{1}{n^2}\,.
\end{align}
This proves \eqref{eq:largestjump:LB_3} because using that
$\delta_o\le\epsilon\mu_o/16$ by definition and that $\epsilon\le 1$
we deduce 
\begin{equation*}
  (\mu_o+4\delta_o)\frac{1-\epsilon}{\mu_o-\delta_o}
  \le 1-\epsilon\frac{5+2\epsilon}{8-\epsilon}\le1-\frac{\epsilon}{2}\,.
\end{equation*}
The proof of \eqref{eq:largestjump:LB} is thus concluded.
\end{proof}

\subsection{Smoothness and quenched CLT}
We now turn to the proofs of Theorems~\ref{thm:fsmooth}
and~\ref{thm:CLT:Ln:quenched}.

\begin{proof}[Proof of Theorem~\ref{thm:fsmooth}]
This result is a straightforward adaptation
of~\cite[Theorem~1.2]{GZ25concentration} in the i.i.d.\ case: since
all the arguments below are a carbon copy of those
in~\cite{GZ25concentration}, we do not reproduce all the details
here. First of all, denoting by $\var_{n,h,\go}$ the variance under
the polymer measure $\prob_{n,h,\go}$, we claim that for all $h>h_c$
\begin{equation}\label{eq:fsmooth:varianceLB}
\liminf_{n\uparrow\infty} \bbE\left[\frac{\var_{n,h,\cdot}[L_n]}{n}\right] >0\,.
\end{equation}
Indeed, this is a direct consequence of the quenched estimates
established by Equations~(46-47) in~\cite{GZ25concentration}, together
with Birkhoff's ergodic theorem and the fact that
$\liminf_{n\uparrow\infty} (1/n)\bbE[\Ex_{n,h,\cdot}[L_n]]>0$ for
$h>h_c$. Moreover, we also have that for every $h\in\R$, $\go\in\gO$,
and integers $n\geq 1$ and $r\geq 2$
(see~\cite[Lemma~4.2]{GZ25concentration})
\begin{equation}\label{eq:cumulants}
\big|\partial^r_h \log Z_{n,h}(\go)\big| \leq 2^r (r!)^2 \sum_{i=0}^{n-1} \sum_{j\in\N} j^r \Ex^{\otimes2}_{j+1,h,\gth^i\go}\bigg[\prod_{k=1}^{j} (1-X_kX_k')\bigg]\,.
\end{equation}
Thereafter, the smoothness of the free energy and the convergence
statements are obtained with an application of the Arzel\`a-Ascoli
theorem to the derivatives of the finite-volume free energy, where
their equicontinuity follows from~\eqref{eq:cumulants},
Lemma~\ref{lem:supercrit:polymercorrelations}, and Birkhoff's ergodic
theorem. Finally, the strict convexity of $f$ is a consequence
of~\eqref{eq:fsmooth:varianceLB} since $\partial_h^2\log
Z_{n,h}=\var_{n,h,\cdot}[L_n]$, and this concludes the proof of the
theorem.
\end{proof}

\begin{proof}[Proof of Theorem~\ref{thm:CLT:Ln:quenched}]
  Exactly as in the i.i.d.\ case, both the concentration inequality
  and the CLT are obtained via the cumulant method for normal
  approximation (see~\cite[Theorem~1.4]{GZ25concentration}
  and~\cite[Theorems~2.4~and~2.5]{DJS22}). In particular, the
  convergence of the variance $\lim_{n\uparrow\infty}
  (1/n)\var_{n,h,\go}[L_n]=\lim_{n\uparrow\infty}
  (1/n)\partial_h^2\log Z_{n,h}(\omega)=\partial^2_h f(h)\eqc v(h)$,
  locally uniformly in $h\in(h_c,+\infty)$ for $\pae$, is a
  consequence of Theorem~\ref{thm:fsmooth}, and a
  \emph{Statulevi\v{c}ius condition} follows from~\eqref{eq:cumulants}
  and Birkhoff's ergodic theorem. Again, these arguments are a carbon
  copy of those in~\cite{GZ25concentration}, so we do not reproduce
  them here.
\end{proof}

\section{The random centering term}
\label{sec:centeringterm}

In this section we investigate the fluctuations of the random
centering term $\Ex_{n,h,\cdot}[L_n]$ in order to prove Theorem
\ref{th:centering}. Part $(i)$ of Theorem \ref{th:centering} is a
consequence of Theorem~\ref{thm:fsmooth} and
Corollary~\ref{corol:correlations:polym}, together with the following
identities, which are Equations (62) and (63) in
\cite{GZ25concentration}: for every $h\in\Rl$, $\omega\in\Omega$, and
integers $0\le a\le m\le b\le n$
  \begin{equation}
    \Ex_{n,h,\omega}[X_a]=\Ex_{m,h,\omega}[X_a]-\frac{\cov_{n,h,\omega}[X_a,X_m]}{\Ex_{n,h,\omega}[X_m]}
    \label{basic_tool_1}
  \end{equation}
  and
  \begin{equation}
    \Ex_{n,h,\omega}[X_b]=\Ex_{n-m,h,\vartheta^m\omega}[X_{b-m}]-\frac{\cov_{n,h,\omega}[X_m,X_b]}{\Ex_{n,h,\omega}[X_m]}\,.
    \label{basic_tool_2}
  \end{equation}
The proof of part $(i)$ of Theorem \ref{th:centering} proceeds exactly
as that of part $(ii)$ of \cite[Theorem~1.3]{GZ25concentration};
therefore we omit the details.

\subsection{Correlations under the environment measure}

The study of the variance of the centering term, as well as the proof
of its CLT, requires the following correlation estimate, which extends
\cite[Corollary~3.1]{GZ25concentration} to our setting.

\begin{lemma}
  \label{lem:correlations:envt}
  For every $H\subset(h_c,+\infty)$ closed there exist $C>0$ and
  $\lambda>0$ such that for all $n\in\N$ and $a,b\in\{0,\ldots, n\}$
  \begin{align*}
    &\sup_{h\in H}\Big|\Exd\big[\Ex_{n,h,\cdot}[X_a]\Ex_{n,h,\cdot}[X_b]\big]-\Exd\big[\Ex_{n,h,\cdot}[X_a]\big]\Exd\big[\Ex_{n,h,\cdot}[X_b]\big]\Big|\\
    &\qquad\qquad\qquad\qquad\qquad\qquad\qquad\le C\sum_{i\in\N}\sum_{j\in\N}\big|\gamma(|i-j|)\big|\ee^{-\lambda|a-i|-\lambda|b-j|}\,.
  \end{align*}
\end{lemma}

\begin{remark}
  \label{rm:correlations:envt}
        {\rm For covariance functions $\gamma$ that decay
          exponentially or super exponentially fast,
          Lemma~\eqref{lem:correlations:envt} ensures exponential
          decay with the distance of the involved covariances. This is
          the case, for instance, of i.i.d.\ charges considered
          in~\cite{GZ25concentration}.  For other choices of $\gamma$,
          the lemma provides a slower decay.  In general, the decay
          with the distance of the covariances in
          Lemma~\eqref{lem:correlations:envt} must be sufficiently
          rapid to yield a concentration bound and a CLT for the
          centering term via the cumulant method used
          in~\cite{GZ25concentration}; this is not guaranteed in our
          setting without imposing strong additional assumptions on
          $\gamma$.}
\end{remark}

\begin{proof}[Proof of Lemma~\ref{lem:correlations:envt}]
Pick a closed set $H\subset(h_c,+\infty)$. Noting that
$-1\le\cov_{n,h,\cdot}[X_a,X_b]\le1$,
Corollary~\ref{corol:correlations:polym} assures us that there exist
positive constants $C$ and $\lambda$ such that for all $n\in\N$ and
$a,b\in\{0,\ldots, n\}$
\begin{equation}
\sup_{h\in H}\bbE\big[\cov_{n,h,\cdot}[X_a,X_b]^2\big] \le\sup_{h\in H}\bbE\big[|\cov_{n,h,\cdot}[X_a,X_b]|\big] \leq C\ee^{-2\lambda|b-a|}\,.
\label{prop:correlations:envt_0}
\end{equation}
The proposition holds with such $C$ and $\lambda$, as we now verify.

\smallskip

Fix $h\in H$, $n\in\N$, and $a,b\in\{0,\ldots, n\}$. Set for brevity
$\Lambda\ceq\Ex_{n,h,\cdot}[X_a]$ and
$\Upsilon\ceq\Ex_{n,h,\cdot}[X_b]$. A standard argument of Gaussian
interpolation and integration by parts (see
\cite[Lemma~6.2]{Bis20notes}) gives
\begin{align*}
  &\Exd[\Lambda\Upsilon]-\Exd[\Lambda]\Exd[\Upsilon]\\
  &\qquad= 
\int_0^1\dd z \sum_{i=1}^n\sum_{j=1}^n\gamma(|i-j|)\int_{\Omega}\probd[\dd\omega]\int_{\Omega}\probd[\dd\omega']\,
\partial_{\omega_i}\Lambda\big(\omega\big)\partial_{\omega_j}\Upsilon\big(z\omega+\sqrt{1-z^2}\omega'\big)\,.
\end{align*}
Observe that
$\partial_{\omega_i}\Lambda(\omega)=\cov_{n,h,\omega}[X_a,X_i]$ and
$\partial_{\omega_j}\Upsilon(\omega)=\cov_{n,h,\omega}[X_b,X_j]$ for
all $i,j\in\{1,\ldots,n\}$ and $\omega\in\Omega$. Also note that,
thanks to well-known properties of Gaussian random variables, for any
$z\in[0,1]$ the sequence $z\omega+\sqrt{1-z^2}\omega'$ is an outcome
of an environment with law $\probd$. Then, the Cauchy-Schwarz
inequality allows us to conclude that
\begin{equation*}
  \big|\Exd[\Lambda\Upsilon]-\Exd[\Lambda]\Exd[\Upsilon]\big|
  \le \sum_{i=1}^n\sum_{j=1}^n\big|\gamma(|i-j|)\big|\sqrt{\Exd\big[\cov_{n,h,\cdot}[X_a,X_i]^2\big]\Exd\big[\cov_{n,h,\cdot}[X_b,X_j]^2\big]}\,.
\end{equation*}
This bound proves the proposition because of \eqref{prop:correlations:envt_0}. 
\end{proof}

The following technical lemma is helpful when applying the bound
introduced in Lemma~\ref{lem:correlations:envt}.
\begin{lemma}
  \label{lem:sums}
  For all finite sets $A,B\subseteq\N$ and for all $\lambda>0$
  \begin{align*}
  &\sum_{a\in A}\sum_{b\in B}\sum_{i\in\N}\sum_{j\in\N}\big|\gamma(|i-j|)\big|\ee^{-\lambda|a-i|-\lambda|b-j|}\\
  &\qquad\le\frac{16\bar\gamma}{(1-\ee^{-\lambda})^3}+\frac{4}{(1-\ee^{-\lambda})^2}\sum_{i=\min A}^{\max A}\sum_{j=\min B}^{\max B}\big|\gamma(|i-j|)\big|\,.
\end{align*}
\end{lemma}

\begin{proof}[Proof of Lemma \ref{lem:sums}]
The idea of the proof is to extend the range of some indices and then
apply a shift in order to disentangle the factors. In particular,
given two finite sets $A,B\subseteq\N$ and a $\lambda>0$, we can split
the range $\N$ of the index $i$ into the components $\{1,\ldots,\min
A-1\}$, $\{\min A,\ldots,\max A\}$, and $\{\max A+1,\max A+2,\ldots\}$
to obtain the bound
\begin{align*}
  &\sum_{a\in A}\sum_{b\in B}\sum_{i\in\N}\sum_{j\in\N}\big|\gamma(|i-j|)\big|\ee^{-\lambda|a-i|-\lambda|b-j|}\\
  &\qquad\le\sum_{a=\min A}^\infty\sum_{i=1}^{\min A-1}\sum_{j\in\mathbb{Z}}\sum_{b\in\mathbb{Z}}\big|\gamma(|i-j|)\big|\ee^{-\lambda(a-i)-\lambda|b-j|}\\
  &\qquad\quad+\sum_{b\in B}\sum_{i=\min A}^{\max A}\sum_{j\in\N}\sum_{a\in\mathbb{Z}}\big|\gamma(|i-j|)\big|\ee^{-\lambda|a-i|-\lambda|b-j|}\\
  &\qquad\quad\quad+\sum_{a=1}^{\max A}\sum_{i=\max A+1}^\infty\sum_{j\in\mathbb{Z}}\sum_{b\in\mathbb{Z}}\big|\gamma(|i-j|)\big|\ee^{-\lambda(i-a)-\lambda|b-j|}\\
  &\qquad\le\frac{8\bar\gamma}{(1-\ee^{-\lambda})^3}+\frac{2}{1-\ee^{-\lambda}}\sum_{b\in B}\sum_{i=\min A}^{\max A}\sum_{j\in\N}\big|\gamma(|i-j|)\big|\ee^{-\lambda|b-j|}\,.
\end{align*}
By repeating a similar procedure for the index $j$ we conclude that
\begin{align*}
  &\sum_{a\in A}\sum_{b\in B}\sum_{i\in\N}\sum_{j\in\N}\big|\gamma(|i-j|)\big|\ee^{-\lambda|a-i|-\lambda|b-j|}\\
  &\qquad\le\frac{8\bar\gamma}{(1-\ee^{-\lambda})^3}+\frac{2}{1-\ee^{-\lambda}}\sum_{b=\min B}^\infty\sum_{j=1}^{\min B-1}\sum_{i\in\mathbb{Z}}\big|\gamma(|i-j|)\big|\ee^{-\lambda(b-j)}\\
  &\qquad\quad+\frac{2}{1-\ee^{-\lambda}}\sum_{i=\min A}^{\max A}\sum_{j=\min B}^{\max B}\sum_{b\in\mathbb{Z}}\big|\gamma(|i-j|)\big|\ee^{-\lambda|b-j|}\\
  &\qquad\quad\quad+\frac{2}{1-\ee^{-\lambda}}\sum_{b=1}^{\max B}\sum_{j=\max B+1}^\infty\sum_{i\in\mathbb{Z}}\big|\gamma(|i-j|)\big|\ee^{-\lambda(j-b)}\\
  &\qquad\le\frac{16\bar\gamma}{(1-\ee^{-\lambda})^3}+\frac{4}{(1-\ee^{-\lambda})^2}\sum_{i=\min A}^{\max A}\sum_{j=\min B}^{\max B}\big|\gamma(|i-j|)\big|\,.
  \qedhere
\end{align*}
\end{proof}

\subsection{Variance of the centering term}

We are now in a position to investigate the variance of the random
centering term. We begin by proving part $(ii)$ of
Theorem~\ref{th:centering}

\begin{proof}[Proof of part $(ii)$ of Theorem~\ref{th:centering}]
We follow the approach of~\cite{GZ25concentration}, specifically of
the proof of part $(iii)$ of Theorem~1.3 therein, exploiting the fact
that most of the arguments remain unchanged in our setting. For any
random variable $\Lambda$ on $\Omega$, we denote for brevity
$\Lambda-\Exd[\Lambda]$ by $\overline{\Lambda}$. The scheme adopted
in~\cite{GZ25concentration} consists of the following three steps:
\begin{itemize}[label=--]
  \item Prove at first by a subadditive argument that the limit
    $\lim_{n\uparrow\infty}
    (1/n)\Exd\big[\overline{\Ex_{2n,h,\cdot}[L_n]}^2\big] \eqc w(h)$
    exists for all $h>h_c$. As noted in~\cite{GZ25concentration}, the
    sequence with $n$-th term
    $\Exd\big[\overline{\Ex_{n,h,\cdot}[L_n]}^2\big]$ does not appear
    to possess a direct subadditivity property. For this reason, it is
    more convenient to first double the system size and deal with
    $\Exd\big[\overline{\Ex_{2n,h,\cdot}[L_n]}^2\big]$ in place of
    $\Exd\big[\overline{\Ex_{n,h,\cdot}[L_n]}^2\big]$.
\item Show later that $\lim_{n\uparrow\infty}
  (1/n)\Exd\big[\overline{\Ex_{n,h,\cdot}[L_n]}^2\big] = w(h)$ for all
  $h> h_c$ exploiting the fact that $\lim_{n\uparrow\infty}
  (1/n)\Exd\big[\overline{\Ex_{2n,h,\cdot}[L_n]}^2\big] = w(h)$.
\item Verify finally that for every closed interval $H\subset
  (h_c,+\infty)$ there exists $c>0$ such that, for all $n\in\N$, the
  function that maps $h\in H$ to
  $\Exd\big[\overline{\Ex_{n,h,\cdot}[L_n]}^2\big]$ is
  $(cn)$-Lipschitz continuous.
\end{itemize}
After resolving the first issue, the remaining two are tackled
in~\cite{GZ25concentration} for an i.i.d.\ disorder through quenched
computations of the form \eqref{basic_tool_1} and
\eqref{basic_tool_2}, together with the application of Lemma~3.4
therein. We have extended the latter to our setting in
Corollary~\ref{corol:correlations:polym}; consequently, a carbon copy
of the proofs in~\cite{GZ25concentration}, which we omit for the sake
of brevity, yields that the same results hold in our setting, provided
we demonstrate that the limit $\lim_{n\uparrow\infty}
(1/n)\Exd\big[\overline{\Ex_{2n,h,\cdot}[L_n]}^2\big] \eqc w(h)$ exists
for all $h> h_c$.  In turn, via an application of the
Arzel\`a-Ascoli theorem, this implies that the function mapping
$h>h_c$ to $\Exd\big[\overline{\Ex_{n,h,\cdot}[L_n]}^2\big]$ converges
to $w$ locally uniformly.  Again, as this follows the same argument as
in~\cite{GZ25concentration}, we do not reproduce it here.

\smallskip

We focus then on proving that the limit $\lim_{n\uparrow\infty}
(1/n)\Exd\big[\overline{\Ex_{2n,h,\cdot}[L_n]}^2\big] \eqc w(h)$ exists
for all $h>h_c$. To this aim, fix some $h>h_c$ and, in order to
simplify the notations, define for $n\in\N$
\begin{equation*}
\nu_n\ceq\Exd\Big[\overline{\Ex_{2n,h,\cdot}[L_n]}^2\Big] \quad \text{and} \quad
\zeta_n\ceq\sum_{a=0}^n(n-a)|\gamma(a)|\,.
\end{equation*}
We shall show below that there exists $C\ge 1$ such that for all
$m,n\in\N$
\begin{equation}
  \label{eq:prop:convergenceofvariance:subadd}
\nu_{m+n}\leq \nu_m + \nu_n + C\big[1+\sqrt{\nu_{\min\{m,n\}}}+\zeta_{m+n}-\zeta_m-\zeta_n\big]\,.
\end{equation}
From this bound and $\nu_1\le 1$ we deduce that $\nu_{n+1}\leq \nu_n
+1+ C(2+\bar\gamma)$ for any $n\in\N$, so in turn $\nu_n\leq
C(2+\bar\gamma)n$ because $C\ge 1$. At the same time we have
$\sqrt{\min\{m,n\}}\le 2(\sqrt{m}+\sqrt{n}-\sqrt{m+n})$. In
conclusion, we see that the sequence
$(\nu_n+C+2C\sqrt{C(2+\bar\gamma)n}-C\zeta_n)_{n\in\N}$ is
subadditive, so the limit
$\lim_{n\uparrow\infty}(\nu_n+C+2C\sqrt{C(2+\bar\gamma)n}-C\zeta_n)/n$
exists. Since also the limit
$\lim_{n\uparrow\infty}\zeta_n/n=\bar\gamma$ exists by the dominated
convergence theorem, we obtain that the limit
$\lim_{n\uparrow\infty}\nu_n/n\eqc w(h)$ exists.

\smallskip

Let us verify bound~\eqref{eq:prop:convergenceofvariance:subadd}.
Pick $m,n\in\N$ and suppose without loss of generality that $m\le
n$. Decomposing $L_{m+n}$ as the sum of $L_n$ and $L_{m+n}-L_n$ we can
write
\begin{align}
  \nonumber
  \nu_{m+n}:=\Exd\Big[\overline{\Ex_{{2m+2n},h,\cdot}[L_{m+n}]}^2\Big]&=\Exd\Big[\overline{\Ex_{2m+2n,h,\cdot}[L_n]}^2\Big]\\
  \nonumber
  &\quad +\Exd\Big[\overline{\Ex_{2m+2n,h,\cdot}[L_{m+n}-L_n]}^2\Big]\\
  &\quad\quad +2\Exd\Big[\overline{\Ex_{2m+2n,h,\cdot}[L_n]}\,\overline{\Ex_{2m+2n,h,\cdot}[L_{m+n}-L_n]}\Big]\,.
\label{eq:var_EL_1}
\end{align}
We claim that there exists a positive constant $C_1$, independent of
$m$ and $n$, such that
\begin{equation}
  \Exd\Big[\overline{\Ex_{2m+2n,h,\cdot}[L_n]}^2\Big] \leq \nu_n + C_1
\label{eq:var_EL_2}
\end{equation}
and
\begin{equation}
  \Exd\Big[\overline{\Ex_{2m+2n,h,\cdot}[L_{m+n}-L_n]}^2\Big] \leq \nu_m + C_1\sqrt{\nu_m} + C_1\,.
\label{eq:var_EL_3}
\end{equation}
In fact, these inequalities are identical to those established in the
proof of part $(iii)$ of~\cite[Theorem 1.3]{GZ25concentration} for the
i.i.d.\ case, which rely solely on quenched computations and Lemma~3.4
therein. As our Corollary~\ref{corol:correlations:polym} extends that
lemma to a Gaussian correlated environment, we are able to recover the
same bounds by directly following the original reasoning
in~\cite{GZ25concentration}.

Regarding the third term on the right-hand side of
\eqref{eq:var_EL_4}, its behavior differs from the i.i.d.\ case when a
correlated environment is considered. Indeed, it is the only term that
requires the use of Lemma~\ref{lem:correlations:envt}, which ensures
the existence of positive constants $C_2$ and $\lambda$, independent
of $m$ and $n$, such that for all $a,b\in\{1,\ldots, 2m+2n\}$
  \begin{align*}
    &\Big|\Exd\big[\Ex_{2m+2n,h,\cdot}[X_a]\Ex_{2m+2n,h,\cdot}[X_b]\big]-\Exd\big[\Ex_{2m+2n,h,\cdot}[X_a]\big]\Exd\big[\Ex_{2m+2n,h,\cdot}[X_b]\big]\Big|\\
    &\qquad\qquad\qquad\qquad\qquad\qquad\qquad\qquad\qquad\le C_2\sum_{i\in\N}\sum_{j\in\N}\big|\gamma(|i-j|)\big|\ee^{-\lambda|a-i|-\lambda|b-j|}\,.
  \end{align*}
So, we find
\begin{align}
  \nonumber
  &\Exd\Big[\overline{\Ex_{2m+2n,h,\cdot}[L_n]}\,\overline{\Ex_{2m+2n,h,\cdot}[L_{m+n}-L_n]}\Big]\\
  \nonumber
  &\qquad=\sum_{a=1}^n\sum_{b=n+1}^{m+n}\Exd\Big[\overline{\Ex_{2m+2n,h,\cdot}[X_a]}\,\overline{\Ex_{2m+2n,h,\cdot}[X_b]}\Big]\\
  \nonumber
  &\qquad\le\sum_{a=1}^n\sum_{b=n+1}^{m+n}\bigg|\Exd\Big[\Ex_{2m+2n,h,\cdot}[X_a]\Ex_{2m+2n,h,\cdot}[X_b]\Big]-\Exd\Big[\Ex_{2m+2n,h,\cdot}[X_a]\Big]
    \Exd\Big[\Ex_{2m+2n,h,\cdot}[X_b]\Big]\bigg|\\
  \nonumber
  &\qquad\le C_2\sum_{a=1}^n\sum_{b=n+1}^{m+n}\sum_{i\in\N}\sum_{j\in\N}\big|\gamma(|i-j|)\big|\ee^{-\lambda |a-i|-\lambda |b-j|}\,,
\end{align}
which thanks to Lemma~\ref{lem:sums} shows that
\begin{align}
  \nonumber
  &\Exd\Big[\overline{\Ex_{2m+2n,h,\cdot}[L_n]}\,\overline{\Ex_{2m+2n,h,\cdot}[L_{m+n}-L_n]}\Big]\\
  \nonumber
  &\qquad\le \frac{16C_2\bar\gamma}{(1-\ee^{-\lambda})^3}+\frac{4C_2}{(1-\ee^{-\lambda})^2}\sum_{i=1}^n\sum_{j=n+1}^{m+n}\big|\gamma(j-i)\big|\,.
\end{align}
On the other hand, we have
\begin{align*}
  \sum_{i=1}^n\sum_{j=n+1}^{n+m}\big|\gamma(j-i)\big|&=\sum_{i=1}^{m+n}\sum_{j=i}^{m+n}\big|\gamma(j-i)\big|-\sum_{i=1}^m\sum_{j=i}^m\big|\gamma(j-i)\big|
  -\sum_{i=1}^n\sum_{j=i}^n\big|\gamma(j-i)\big|\\
  &=\sum_{a=0}^{m+n}(m+n-a)|\gamma(a)|-\sum_{a=0}^m(m-a)|\gamma(a)|-\sum_{a=0}^n(n-a)|\gamma(a)|\\
  &=\zeta_{m+n}-\zeta_m-\zeta_n\,,
\end{align*}
which then allows us to obtain
\begin{align}
  \nonumber
  &\Exd\Big[\overline{\Ex_{2m+2n,h,\cdot}[L_n]}\,\overline{\Ex_{2m+2n,h,\cdot}[L_{m+n}-L_n]}\Big]\\
  &\qquad\le \frac{16C_2\bar\gamma }{(1-\ee^{-\lambda})^3}+\frac{4C_2}{(1-\ee^{-\lambda})^2}\big[\zeta_{m+n}-\zeta_m-\zeta_n\big]\,.
\label{eq:var_EL_4}
\end{align}

In conclusion, plugging \eqref{eq:var_EL_2}, \eqref{eq:var_EL_3}, and
\eqref{eq:var_EL_4} in \eqref{eq:var_EL_1}, and also bearing in mind
that the difference $\zeta_{m+n}-\zeta_m-\zeta_n$ is non negative, we find
\eqref{eq:prop:convergenceofvariance:subadd} with, e.g.,
\begin{equation*}
  C\ceq 1+2C_1+\frac{8C_2}{(1-\ee^{-\lambda})^2}+\frac{32C_2\bar\gamma }{(1-\ee^{-\lambda})^3}\,.
  \qedhere
\end{equation*}
\end{proof}

The final task of this section is to verify part $(iv)$ of
Theorem~\ref{th:centering}.

\begin{proof}[Proof of part $(iv)$ of Theorem \ref{th:centering}]
  Fix $h>h_c$ and suppose that $\gamma(0)+2\sum_{a\in\N}\gamma(a)>0$,
  otherwise there is nothing to prove since $w(h)\ge 0$ by
  definition. We prove that on the one hand
  \begin{equation}
    \limsup_{n\uparrow\infty}\frac{1}{n} \int_\Omega \sum_{a=1}^n\omega_a \Ex_{n,h,\omega}[L_n]\,\probd[\dd\omega]\le\sqrt{\bigg[\gamma(0)+2\sum_{a\in\N}\gamma(a)\bigg]w(h)}\,,
\label{eq:w_pos_1}
  \end{equation}
  and on the other hand
  \begin{equation}
    \liminf_{n\uparrow\infty}\frac{1}{n} \int_\Omega \sum_{a=1}^n\omega_a \Ex_{n,h,\omega}[L_n]\,\probd[\dd\omega]\ge\bigg[\gamma(0)+2\sum_{a\in\N}\gamma(a)\bigg]v(h)\,.
\label{eq:w_pos_2}
  \end{equation}
These bounds imply part $(iv)$ of Theorem~\ref{th:centering}.

\smallskip

Recall that the charges are centered and denote by
$\overline{\Lambda}$ the difference $\Lambda-\Exd[\Lambda]$, where
$\Lambda$ is any random variable $\Lambda$ on $\Omega$. Bound
\eqref{eq:w_pos_1} is due to the Cauchy-Schwarz inequality, which
gives for all $n\in\N$
\begin{align*}
  \int_\Omega \sum_{a=1}^n\omega_a \Ex_{n,h,\omega}[L_n]\,\probd[\dd\omega]&=
  \int_\Omega \sum_{a=1}^n\omega_a\, \overline{\Ex_{n,h,\omega}[L_n]}\,\probd[\dd\omega]\\
  &\le \sqrt{\int_\Omega \bigg(\sum_{a=1}^n\omega_a\bigg)^{\!2}\,\probd[\dd\omega]}\sqrt{\Exd\Big[\overline{\Ex_{n,h,\omega}[L_n]}^2\Big]}\,.
\end{align*}
Dividing by $n$ and letting $n\uparrow\infty$ we get
\eqref{eq:w_pos_1} since, on the one hand, $\lim_{n\uparrow\infty}
(1/n)\Exd\big[\overline{\Ex_{n,h,\cdot}[L_n]}^2\big] = w(h)$ by part
$(ii)$ of Theorem \ref{th:centering} and, on the other hand, one has
$\int_\Omega
\big(\sum_{a=1}^n\omega_a\big)^2\,\probd[\dd\omega]=\sum_{a=1}^n\sum_{b=1}^n\gamma(|a-b|)=n\gamma(0)+2\sum_{a=1}^n(n-a)\gamma(a)$
for any $n\in\N$ as noted in the introduction.

\smallskip

Coming to bound \eqref{eq:w_pos_2}, the Gaussian integration by parts
formula (see~\cite[Lemma~5.2]{Bis20notes}) yields for all $n\in\N$
\begin{align}
  \nonumber
  \int_\Omega \sum_{a=1}^n\omega_a \Ex_{n,h,\omega}[L_n]\,\probd[\dd\omega]&=
 \sum_{a=1}^n\sum_{b=1}^n \gamma(|a-b|)\int_\Omega\partial_{\omega_b} \Ex_{n,h,\omega}[L_n]\,\probd[\dd\omega]\\
 &=  \sum_{a=1}^n\sum_{b=1}^n\sum_{c=1}^n \gamma(|a-b|)\,\Exd\big[\cov_{n,h,\cdot}[X_b,X_c]\big] \,,
\label{eq:w_pos_3}
\end{align}
where the second equality is due to the identities
$L_n=\sum_{c=1}^nX_c$ and
$\partial_{\omega_b}\Ex_{n,h,\omega}[X_c]=\cov[X_b,X_c]$.  Noting that
$\sum_{a=1}^n\sum_{b=1}^n\cov_{n,h,\cdot}[X_a,X_b]=\var_{n,h,\cdot}[L_n]$
and introducing the numerical factor
$\zeta_n:=\sum_{a=n}^\infty\gamma(a)$ with $n\in\N$, we can recast
the right-hand side of \eqref{eq:w_pos_3} in order to obtain
\begin{align}
  \nonumber
  \int_\Omega \sum_{a=1}^n\omega_a \Ex_{n,h,\omega}[L_n]\,\probd[\dd\omega]&=
 \bigg[\gamma(0)+2\sum_{a\in\N}\gamma(a)\bigg]\Exd\big[\var_{n,h,\cdot}[L_n]\big]\\
 &\quad-\sum_{a=1}^n\sum_{b=1}^n\zeta_a\Big(\Exd\big[\cov_{n,h,\cdot}[X_a,X_b]\big] +\Exd\big[\cov_{n,h,\cdot}[X_{n-a+1},X_b]\big] \Big)\,.
\label{eq:w_pos_4}
\end{align}
Now we observe that $|\zeta_n|\le\bar\gamma$ for any $n\in\N$ and that
$\lim_{n\uparrow\infty}\zeta_n=0$, so in particular, given an
arbitrary $\eps>0$, there exists $n_o\in\N$ such that
$|\zeta_n|\le\eps$ for all $n\ge n_o$. Moreover,
Corollary~\ref{corol:correlations:polym} states that there exist
positive constants $C$ and $\lambda$ such that for all integers $1\le
a\le b\le n$
\begin{equation*}
\bbE\big[|\cov_{n,h,\cdot}[X_a,X_b]|\big] \leq C \ee^{-\lambda (b-a)}\,.
\end{equation*}
Combining these facts with \eqref{eq:w_pos_4}, we find for all
integers $n>n_o$
\begin{align*}
  \int_\Omega \sum_{a=1}^n\omega_a \Ex_{n,h,\omega}[L_n]\,\probd[\dd\omega]&\ge
 \bigg[\gamma(0)+2\sum_{a\in\N}\gamma(a)\bigg]\Exd\big[\var_{n,h,\cdot}[L_n]\big]\\
 &\quad-2\bar\gamma C\sum_{a=1}^{n_o}\sum_{b=1}^n \ee^{-\lambda|a-b|}-2\eps C\sum_{a=n_o+1}^n\sum_{b=1}^n \ee^{-\lambda|a-b|}\,.
\end{align*}
Dividing by $n$ and letting $n\uparrow\infty$, we finally realize that
\begin{equation*}
  \liminf_{n\uparrow\infty}\frac{1}{n}\int_\Omega \sum_{a=1}^n\omega_a
  \Ex_{n,h,\omega}[L_n]\,\probd[\dd\omega]\ge
  \bigg[\gamma(0)+2\sum_{a\in\N}\gamma(a)\bigg]v(h)
  -\frac{4C}{1-\ee^{-\lambda}}\,\eps\,
\end{equation*}
as Theorem~\ref{thm:fsmooth} implies $\lim_{n\uparrow\infty}
(1/n)\Exd[\var_{n,h,\cdot}[L_n]]=\lim_{n\uparrow\infty}
(1/n)\Exd[\partial_h^2\log Z_{n,h}]=\partial^2_h f(h)\eqc v(h)$.  This
inequality proves \eqref{eq:w_pos_2} because $\eps$ can be arbitrarily
small.
\end{proof}

\subsection{CLT for the centering term}
Finally we establish the CLT for the centering term, namely part
$(iii)$ of Theorem~\ref{th:centering}, by means of the Bernstein
blocking technique, which allows one to reduce a problem of dependent
variables to an independent scenario. For $n\in\N$ we devise $r_n$
blocks of size $m_n$ spaced by $d_n$ sites with $m_n\ceq \lfloor
n^{1/5}\rfloor$, $d_n\ceq \lfloor n^{1/6}\rfloor$, and $r_n\ceq\lfloor
n/(m_n+d_n)\rfloor$. Concretely, given a compact set $H\subset
(h_c,+\infty)$ and supposing that $w(h)>0$ for all $h\in H$, for $h\in
H$ and $s\in\{0,\ldots,r_{n}-1\}$ we consider the block random
variable
\begin{equation*}
  \Delta_{n,h,s}\ceq\frac{1}{\sqrt{nw(h)}}\sum_{a=l_s+1}^{l_s+m_n}\Ex_{n,h,\cdot}[X_a]\,,
\end{equation*}
where $l_s\ceq s(m_n+d_n)$. As the number $d_n$ is small compared with
$m_n$, we expect that
$\Ex_{n,h,\cdot}[L_n]/\sqrt{nw(h)}=\sum_{a=1}^n\Ex_{n,h,\cdot}[X_a]/\sqrt{nw(h)}$
is close to the sum $\sum_{s=0}^{r_n-1}\Delta_{n,h,s}$ for large
$n$. At the same time, any two random variables $\Delta_{n,h,s}$ and
$\Delta_{n,h,\sigma}$ with $s\ne \sigma$ are separated by at least
$d_n$ sites, so they are almost independent at large $n$.

\smallskip

Let us develop a key bound to make the above ideas fruitful.  Define
\begin{equation*}
  \Upsilon_{n,h}\ceq\frac{\Ex_{n,h,\cdot}[L_n]}{\sqrt{nw(h)}}-\sum_{s=0}^{r_n-1}\Delta_{n,h,s}
\end{equation*}
and note that, putting $l_{r_n}\ceq n$, one has
\begin{equation}
  \Upsilon_{n,h}=\frac{1}{\sqrt{nw(h)}}\sum_{s=0}^{r_n-1}\sum_{a=l_s+m_n+1}^{l_{s+1}}\Ex_{n,h,\cdot}[X_a]\,.
\label{eq:Ups_formula}
\end{equation}
Furthermore, in order to simplify upcoming notation, for any random
variable $\Lambda$ on $\Omega$ denote $\Lambda-\Exd[\Lambda]$ by
$\overline{\Lambda}$.  Applying Esseen's smoothing inequality (see
\cite[§XVI.3, Lemma~2]{Fel66book}) we can state that for all $n\in\N$,
$h\in H$, $u\in\Rl$, and $\epsilon>0$
\begin{align*}
&\Bigg|\probd\Big[\overline{\Ex_{n,h,\cdot}[L_n]}\le \sqrt{n w(h)}\,u\Big]
  -\frac{1}{\sqrt{2\pi}}\int_{-\infty}^u\ee^{-\frac{1}{2}z^2}\dd z\Bigg|\\
  &\qquad\le
  \int_{-1/\epsilon}^{1/\epsilon}\bigg|\frac{\Exd\big[\ee^{\ii\zeta\sum_{s=0}^{r_n-1}\overline{\Delta_{n,h,s}}+\ii\zeta\Upsilon_{n,h}}\big]-\ee^{-\frac{1}{2}\zeta^2}}{\zeta}\bigg|\dd\zeta+4\epsilon\\
  &\qquad\le \int_{-1/\epsilon}^{1/\epsilon}\bigg|\frac{\Exd\big[\ee^{\ii\zeta\sum_{s=0}^{r_n-1}\overline{\Delta_{n,h,s}}}\big]-\ee^{-\frac{1}{2}\zeta^2}}{\zeta}\bigg|\dd\zeta
  +\int_{-1/\epsilon}^{1/\epsilon}\frac{\Exd\big[|\ee^{\ii\zeta\,\overline{\Upsilon_{n,h}}}-1|\big]}{|\zeta|}\dd\zeta+4\epsilon\,.
\end{align*}
On the other hand, using Markov's inequality and $|\ee^{\ii z}-1|\le \min\{2,|z|\}$
for all $z\in\Rl$ (see \cite{Dur10book}, Lemma 3.3.7) we see that
\begin{equation*}
  \Exd\Big[\big|\ee^{\ii\zeta \overline{\Upsilon_{n,h}}}-1\big|\Big]
  \le 2\probd\Big[|\zeta|\overline{\Upsilon_{n,h}}^2\ge \epsilon^3\Big]+\sqrt{|\zeta|\epsilon^3}
  \le \frac{2|\zeta|}{\epsilon^3}\,\Exd\big[\overline{\Upsilon_{n,h}}^2\big]+\sqrt{|\zeta|\epsilon^3}\,,
\end{equation*}
so
\begin{align*}
&\Bigg|\probd\Big[\overline{\Ex_{n,h,\cdot}[L_n]}\le \sqrt{n w(h)}\,u\Big]
  -\frac{1}{\sqrt{2\pi}}\int_{-\infty}^u\ee^{-\frac{1}{2}z^2}\dd z\Bigg|\\
  &\qquad\le \int_{-1/\epsilon}^{1/\epsilon}\bigg|\frac{\Exd\big[\ee^{\ii\zeta\sum_{s=0}^{r_n-1}\overline{\Delta_{n,h,s}}}\big]-\ee^{-\frac{1}{2}\zeta^2}}{\zeta}\bigg|\dd\zeta
  +\frac{4}{\epsilon^4}\,\Exd\big[\overline{\Upsilon_{n,h}}^2\big]+8\epsilon\,.
\end{align*}
At this point, the identity
\begin{align*}
   &\Exd\Big[\ee^{\ii\zeta\sum_{s=0}^{r_n-1}\overline{\Delta_{n,h,s}}}\Big]-\prod_{s=0}^{r_n-1}\Exd\Big[\ee^{\ii\zeta\overline{\Delta_{n,h,s}}}\Big]\\
  &\qquad =\sum_{s=1}^{r_n-1}\Bigg\{\Exd\Big[\ee^{\ii\zeta\sum_{\sigma=0}^s\overline{\Delta_{n,h,\sigma}}}\Big]\prod_{\sigma=s+1}^{r_n}\Exd\Big[\ee^{\ii\zeta\overline{\Delta_{n,h,\sigma}}}\Big]-
  \Exd\Big[\ee^{\ii\zeta\sum_{\sigma=0}^{s-1}\overline{\Delta_{n,h,\sigma}}}\Big]\prod_{\sigma=s}^{r_n}\Exd\Big[\ee^{\ii\zeta\overline{\Delta_{n,h,\sigma}}}\Big]\Bigg\}
  \end{align*}
gives
\begin{align}
  \nonumber
&\Bigg|\probd\Big[\overline{\Ex_{n,h,\cdot}[L_n]}\le \sqrt{n w(h)}\,u\Big]
  -\frac{1}{\sqrt{2\pi}}\int_{-\infty}^u\ee^{-\frac{1}{2}z^2}\dd z\Bigg|\\
  \nonumber
  &\qquad\le \int_{-1/\epsilon}^{1/\epsilon}\bigg|\frac{\prod_{s=0}^{r_n-1}\Exd\big[\ee^{\ii\zeta\overline{\Delta_{n,h,s}}}\big]-\ee^{-\frac{1}{2}\zeta^2}}{\zeta}\bigg|\dd\zeta
  +\frac{4}{\epsilon^4}\,\Exd\big[\overline{\Upsilon_{n,h}}^2\big]+8\epsilon\\
  &\qquad\quad+\int_{-1/\epsilon}^{1/\epsilon}\sum_{s=1}^{r_n-1}
  \bigg|\frac{\Exd\big[\ee^{\ii\zeta\sum_{\sigma=0}^s\overline{\Delta_{n,h,\sigma}}}\big]-\Exd\big[\ee^{\ii\zeta\sum_{\sigma=0}^{s-1}\overline{\Delta_{n,h,\sigma}}}\big]
  \Exd\big[\ee^{\ii\zeta\overline{\Delta_{n,h,s}}}\big]}{\zeta}\bigg|\dd\zeta\,.
\label{eq:CLT_centering_10}
\end{align}
This bound is the key step towards demonstrating part $(iii)$ of
Theorem~\ref{th:centering}. We shall analyze the arguments of the two
integrals in the right-hand side of \eqref{eq:CLT_centering_10},
verifying what follows.  According to
Corollary~\ref{corol:correlations:polym} and
Lemma~\ref{lem:correlations:envt}, there exist positive constants $C$
and $\lambda$ such that for all $n\in\N$ and $a,b\in\{0,\ldots, n\}$
\begin{equation}
\sup_{h\in H}\Exd\big[\cov_{n,h,\cdot}[X_a,X_b]^2\big] \le\sup_{h\in H}\Exd\big[|\cov_{n,h,\cdot}[X_a,X_b]|\big] \leq C(1-\ee^{-\lambda})^3\ee^{-2\lambda|b-a|}
\label{eq:CLT_centering_1}
\end{equation}
and
\begin{align}
  \nonumber
    &\sup_{h\in H}\Big|\Exd\big[\Ex_{n,h,\cdot}[X_a]\Ex_{n,h,\cdot}[X_b]\big]-\Exd\big[\Ex_{n,h,\cdot}[X_a]\big]\Exd\big[\Ex_{n,h,\cdot}[X_b]\big]\Big|\\
    &\qquad\qquad\qquad\qquad\qquad\qquad\le C(1-\ee^{-\lambda})^3\sum_{i\in\N}\sum_{j\in\N}\big|\gamma(|i-j|)\big|\ee^{-\lambda|a-i|-\lambda|b-j|}\,.
\label{eq:CLT_centering_2}
\end{align}
Then, with the constants $C$ and $\lambda$ at hand, we shall prove that for every
$n\in\N$, $h\in H$, and $\zeta\in\Rl$
\begin{equation}
  \bigg|\prod_{s=0}^{r_n-1}\Exd\Big[\ee^{\ii\zeta\overline{\Delta_{n,h,s}}}\Big]
  -\ee^{-\frac{1}{2}\zeta^2}\bigg|\le \frac{2|\zeta|^3}{\sqrt{w(h)^3}}\frac{m_n^2}{\sqrt{n}}
  +\frac{\zeta^2}{2}\bigg|\sum_{s=0}^{r_n-1}\Exd\big[\overline{\Delta_{n,h,s}}^2\big]-1\bigg|\,,
\label{eq:CLT_centering_11}
\end{equation}
with
\begin{align}
  \nonumber
  \bigg|\sum_{s=0}^{r_n-1}\Exd\big[\overline{\Delta_{n,h,s}}^2\big]-1\bigg|
  &\le\bigg|\frac{\Exd\big[\overline{\Ex_{n,h,\cdot}[L_n]}^2\big]}{nw(h)}-1\bigg|\\
  &\quad+\frac{32C\bar\gamma}{w(h)}\frac{r_n}{n}  +\frac{24C\bar\gamma}{w(h)}\frac{n-r_nm_n}{n}
  +\frac{8C}{w(h)}\sum_{a=d_n+1}^\infty|\gamma(a)|\,,
   \label{eq:CLT_centering_12}
\end{align}
and
\begin{align}
  \nonumber
  &\sum_{s=1}^{r_n-1}\Big|\Exd\big[\ee^{\ii\zeta\sum_{k=0}^s\overline{\Delta_{n,h,\sigma}}}\big]-\Exd\big[\ee^{\ii\zeta\sum_{k=0}^{s-1}\overline{\Delta_{n,h,\sigma}}}\big]
  \Exd\big[\ee^{\ii\zeta\overline{\Delta_{n,h,s}}}\big]\Big|\\
&\qquad\qquad\qquad\qquad\qquad\qquad\qquad\le  \frac{16C\bar\gamma\zeta^2}{w(h)}\frac{r_n}{n}+\frac{4C\zeta^2}{w(h)}\sum_{a=d_n+1}^\infty|\gamma(a)|\,.
\label{eq:CLT_centering_13}
\end{align}
Furthermore, we also claim that
\begin{equation}
\Exd\big[\overline{\Upsilon_{n,h}}^2\big]\le\frac{8C\bar\gamma}{w(h)}\frac{n-r_nm_n}{n},
  \label{eq:CLT_centering_14}
\end{equation}
and, since $w$ is continuous on $(h_c,+\infty)$, there exists $h_o\in
H$ such that $w(h)\ge w(h_o)>0$ for all $h\in H$.  Combining
\eqref{eq:CLT_centering_10} with \eqref{eq:CLT_centering_11},
\eqref{eq:CLT_centering_12}, \eqref{eq:CLT_centering_13}, and
\eqref{eq:CLT_centering_14} we conclude that for all $n\in\N$ and
$\epsilon>0$
\begin{align*}
&\sup_{h\in H}\,\sup_{u\in\Rl}\,\Bigg|\probd\Big[\overline{\Ex_{n,h,\cdot}[L_n]}\le \sqrt{n w(h)}\,u\Big]
  -\frac{1}{\sqrt{2\pi}}\int_{-\infty}^u\ee^{-\frac{1}{2}z^2}\dd z\Bigg|\\
  &\qquad\le \frac{2C\bar\gamma}{\epsilon^3\sqrt{w(h_o)^3}}\frac{m_n^2}{\sqrt{n}}
  +\frac{1}{2\epsilon^2w(h_o)}\sup_{h\in H}\bigg|\frac{\Exd\big[\overline{\Ex_{n,h,\cdot}[L_n]}^2\big]}{n}-w(h)\bigg|\\
  &\qquad\quad+\bigg(\frac{12}{\epsilon^2}+\frac{32}{\epsilon^4}\bigg)\frac{C\bar\gamma}{w(h_o)}\frac{n-r_nm_n}{n}+8\epsilon
  + \frac{32C\bar\gamma}{\epsilon^2w(h_o)}\frac{r_n}{n}+\frac{8C}{\epsilon^2w(h_o)} \sum_{a=d_n+1}^\infty|\gamma(a)|\,.
\end{align*}
Recalling part $(ii)$ of Theorem~\ref{th:centering}, from here we get
part $(iii)$ by letting first $n\uparrow\infty$ and then
$\epsilon\downarrow 0$ due to the suitable choice of $m_n$ and $d_n$. It only remains to prove~\eqref{eq:CLT_centering_11} through~\eqref{eq:CLT_centering_14}.

\begin{proof}[Proof of \eqref{eq:CLT_centering_11}]
  Observe that $0\le \Delta_{n,h,s}\le m_n/\sqrt{nw(h)}$ for every
  $n\in\N$, $h\in H$, and $s\in\{0,\ldots,r_{n-1}\}$.  Since
  $|\ee^{\ii z}-1-\ii z+z^2/2|\le |z|^3$ for all $z\in\Rl$ (see
  \cite{Dur10book}, Lemma 3.3.7), for $n\in\N$, $h\in H$, $\zeta\in\Rl$,
  and $s\in\{0,\ldots,r_n-1\}$ we can write
\begin{equation*}
  \bigg|\Exd\big[\ee^{\ii\zeta\overline{\Delta_{n,h,s}}}\big]-1+
  \frac{\zeta^2}{2}\Exd\big[\overline{\Delta_{n,h,s}}^2\big]\bigg|\le \frac{|\zeta|^3m_n^3}{\sqrt{n^3w(h)^3}}\,.
\end{equation*}
As $|\ee^{-z^2}-1+z^2|\le 2|z|^3$ for all $z\in\Rl$, we can also
state that
\begin{equation*}
  \bigg|\Exd\Big[\ee^{\ii\zeta\overline{\Delta_{n,h,s}}}\Big]
  -\ee^{-\frac{\zeta^2}{2}\Exd[\overline{\Delta_{n,h,s}}^2]}\bigg|\le \frac{2|\zeta|^3m_n^3}{\sqrt{n^3w(h)^3}}\,.
\end{equation*}
This inequality gives
\begin{equation}
  \bigg|\prod_{s=0}^{r_n-1}\Exd\Big[\ee^{\ii\zeta\overline{\Delta_{n,h,s}}}\Big]
  -\ee^{-\frac{\zeta^2}{2}\sum_{s=0}^{r_n-1}\Exd[\overline{\Delta_{n,h,s}}^2]}\bigg|\le \frac{2|\zeta|^3m_n^3r_n}{\sqrt{n^3w(h)^3}}
  \le \frac{2|\zeta|^3}{\sqrt{w(h)^3}}\frac{m_n^2}{\sqrt{n}}
 \label{eq:CLT_centering_111}
\end{equation}
because, on the one hand,
$\big|\prod_{s=0}^{r-1}z_s-\prod_{s=0}^{r-1}z_s'\big|\le
\sum_{s=0}^{r-1}|z_s-z_s'|$ for every $r\in\N$ and complex numbers
$z_1,\ldots,z_r$ and $z_1',\ldots,z_r'$ with modulus smaller than or
equal to one (see \cite{Dur10book}, Lemma 3.4.3) and, on the other hand,
$r_n\le n/m_m$. Bound \eqref{eq:CLT_centering_11} follows from
\eqref{eq:CLT_centering_111} thanks to the fact that
$|\ee^{-z}-\ee^{-z'}|\le|z-z'|$ for all non-negative $z,z'\in\Rl$.
\end{proof}

\begin{proof}[Proof of \eqref{eq:CLT_centering_14}]
Expliciting $\Upsilon_{n,h}$ given by formula \eqref{eq:Ups_formula}
and invoking \eqref{eq:CLT_centering_2} we see that for all $n\in\N$
and $h\in H$
\begin{align*}
  &\Exd\big[\overline{\Upsilon_{n,h}}^2\big]\\
  &\qquad\le\frac{1}{nw(h)}\sum_{s=0}^{r_n-1}\sum_{a=l_s+m_n+1}^{l_{s+1}}\sum_{b=1}^n
  \Big|\Exd\big[\Ex_{n,h,\cdot}[X_a]\Ex_{n,h,\cdot}[X_b]\big]-\Exd\big[\Ex_{n,h,\cdot}[X_a]\big]\Exd\big[\Ex_{n,h,\cdot}[X_b]\big]\Big|\\
  &\qquad\le\frac{C(1-\ee^{-\lambda})^3}{nw(h)}
  \sum_{s=0}^{r_n-1}\sum_{a=l_s+m_n+1}^{l_{s+1}}\sum_{i\in\mathbb{Z}}\sum_{j\in\mathbb{Z}}\sum_{b\in\mathbb{Z}}\big|\gamma(|i-j|)\big|\ee^{-\lambda|a-i|-\lambda|b-j|}\\
  &\qquad\le\frac{8C\bar\gamma}{nw(h)}\sum_{s=0}^{r_n-1}(l_{s+1}-l_s-m_n)=\frac{8C\bar\gamma}{w(h)}\frac{n-r_nm_n}{n}\,.
\qedhere
\end{align*}
\end{proof}

\begin{proof}[Proof of \eqref{eq:CLT_centering_12}]
For all $n\in\N$ and $h\in H$ we have
\begin{align}
  \nonumber
  \frac{\Exd\big[\overline{\Ex_{n,h,\cdot}[L_n]}^2\big]}{nw(h)}&= \Exd\bigg[\bigg(\sum_{s=0}^{r_n-1}\overline{\Delta_{n,h,s}}+\overline{\Upsilon_{n,h}}\bigg)^{\!\!2}\,\bigg]\\
  \nonumber
  &=\sum_{s=0}^{r_n-1}\Exd\big[\overline{\Delta_{n,h,s}}^2\big]+2\sum_{s=0}^{r_n-2}\sum_{\sigma=s+1}^{r_n-1}\Exd\big[\overline{\Delta_{n,h,s}}\,\overline{\Delta_{n,h,\sigma}}\big]\\
  &\qquad\qquad\qquad\qquad\quad+2\Exd\bigg[\sum_{s=0}^{r_n-1}\overline{\Delta_{n,h,s}}\,\overline{\Upsilon_{n,h}}\bigg]+ \Exd\big[\overline{\Upsilon_{n,h}}^2\big]\,.
\label{eq:CLT_centering_123}
\end{align}
Similarly to \eqref{eq:CLT_centering_14} we find
\begin{align}
  \nonumber
  &\Bigg|\Exd\bigg[\sum_{s=0}^{r_n-1}\overline{\Delta_{n,h,s}}\,\overline{\Upsilon_{n,h}}\bigg]\Bigg|\\
  \nonumber
  &\qquad\le \frac{1}{nw(h)}\sum_{s=0}^{r_n-1}\sum_{a=1}^n\sum_{b=l_s+m_n+1}^{l_{s+1}}
  \Big|\Exd\big[\Ex_{n,h,\cdot}[X_a]\Ex_{n,h,\cdot}[X_b]\big]-\Exd\big[\Ex_{n,h,\cdot}[X_a]\big]\Exd\big[\Ex_{n,h,\cdot}[X_b]\big]\Big|\\
  &\qquad\le \frac{8C\bar\gamma}{w(h)}\frac{n-r_nm_n}{n}\,.
  \label{eq:CLT_centering_124}
\end{align}
Moreover, \eqref{eq:CLT_centering_2} and Lemma \ref{lem:sums} show
that
\begin{align*}
  &\Bigg|\sum_{s=0}^{r_n-2}\sum_{\sigma=s+1}^{r_n-1}\Exd\big[\overline{\Delta_{n,h,s}}\,\overline{\Delta_{n,h,\sigma}}\big]\Bigg|\\
  &\qquad\le\frac{1}{nw(h)}\sum_{s=0}^{r_n-2}\sum_{a=l_s+1}^{l_s+m_n}\sum_{b=l_{s+1}+1}^n
  \Big|\Exd\big[\Ex_{n,h,\cdot}[X_a]\Ex_{n,h,\cdot}[X_b]\big]-\Exd\big[\Ex_{n,h,\cdot}[X_a]\big]\Exd\big[\Ex_{n,h,\cdot}[X_b]\big]\Big|\\
  &\qquad\le \frac{C(1-\ee^{-\lambda})^3}{nw(h)}
  \sum_{s=0}^{r_n-2}\sum_{a=l_s+1}^{l_s+m_n}\sum_{b=l_{s+1}+1}^n\sum_{i\in\N}\sum_{j\in\N}\big|\gamma(|i-j|)\big|\ee^{-\lambda|a-i|-\lambda(b-j)}\\
  &\qquad\le \frac{16C\bar\gamma}{w(h)}\frac{r_n}{n}
  +\frac{4C}{nw(h)}\sum_{s=0}^{r_n-2}\sum_{i=l_s+1}^{l_s+m_n}\sum_{j=l_{s+1}+1}^n|\gamma(j-i)|\,.
\end{align*}
From here, since $l_{s+1}=l_s+m_n+d_n$ for $s\le r_n-2$ and since
$r_nm_n\le n$, we deduce that
\begin{equation}
  \Bigg|\sum_{s=0}^{r_n-2}\sum_{\sigma=s+1}^{r_n-1}\Exd\big[\overline{\Delta_{n,h,s}}\,\overline{\Delta_{n,h,\sigma}}\big]\Bigg|
  \le \frac{16C\bar\gamma}{w(h)}\frac{r_n}{n}+\frac{4C}{w(h)}\sum_{a=d_n+1}^\infty|\gamma(a)|\,.
\label{eq:CLT_centering_125}
\end{equation}
Plugging \eqref{eq:CLT_centering_125}, \eqref{eq:CLT_centering_124}, and
\eqref{eq:CLT_centering_14} in \eqref{eq:CLT_centering_123} we finally
obtain for all $n\in\N$ and $h\in H$
\begin{align*}
  &\Bigg|\sum_{s=0}^{r_n-1}\Exd\big[\overline{\Delta_{n,h,s}}^2\big]-\frac{\Exd\big[\overline{\Ex_{n,h,\cdot}[L_n]}^2\big]}{nw(h)}\Bigg|\\
  &\qquad\le\frac{32C\bar\gamma}{w(h)}\frac{r_n}{n}  +\frac{24C\bar\gamma}{w(h)}\frac{n-r_nm_n}{n}
  +\frac{8C}{w(h)}\sum_{a=d_n+1}^\infty|\gamma(a)|\,,
\end{align*}
which directly implies bound \eqref{eq:CLT_centering_12}.
\end{proof}

\begin{proof}[Proof of \eqref{eq:CLT_centering_13}]
To begin with, we observe that for all $n\in\N$, $h\in H$,
$i\in\{1,\ldots,n\}$, $s\in\{0,\ldots,r_n-1\}$, and
$\omega=(\omega_a)_{a\in\N_0}$
\begin{equation*}
\partial_{\omega_i}\overline{\Delta_{n,h,s}}(\omega)=\frac{1}{\sqrt{nw(h)}}\sum_{a=l_s+1}^{l_s+m_n}\cov_{n,h,\omega}[X_a,X_i]\,.
\end{equation*}
Using this, a standard argument of Gaussian interpolation and
integration by parts (see \cite[Lemma~6.2]{Bis20notes}) gives for all
$n\in\N$, $h\in H$, $\zeta\in\Rl$, and $s\in\{1,\ldots,r_n-1\}$
\begin{align*}
  &\Exd\big[\ee^{\ii\zeta\sum_{\sigma=0}^s\overline{\Delta_{n,h,\sigma}}}\big]-\Exd\big[\ee^{\ii\zeta\sum_{\sigma=0}^{s-1}\overline{\Delta_{n,h,\sigma}}}\big]
  \Exd\big[\ee^{\ii\zeta\overline{\Delta_{n,h,s}}}\big]\\
  &\qquad= -\frac{\zeta^2}{nw(h)}
  \int_0^1\dd z \sum_{i=1}^n\sum_{j=1}^n\gamma(|i-j|)\int_{\Omega}\probd[\dd\omega]\int_{\Omega}\probd[\dd\omega']\\
  &\qquad\quad\times\bigg(\sum_{\sigma=0}^{s-1}\sum_{a=l_\sigma+1}^{l_\sigma+m_n}\cov_{n,h,\omega}[X_a,X_i]\bigg)\ee^{\ii\zeta\sum_{\sigma=0}^{s-1}\overline{\Delta_{n,h,\sigma}}(\omega)}\\
&\qquad\quad\quad\times\bigg(\sum_{a=l_s+1}^{l_s+m_n}\cov_{n,h,z\omega+\sqrt{1-z^2}\omega'}[X_a,X_i]\bigg)\ee^{\ii\zeta\overline{\Delta_{n,h,s}}(z\omega+\sqrt{1-z^2}\omega')}\,,
\end{align*}
so
\begin{align}
  \nonumber
  &\Big|\Exd\big[\ee^{\ii\zeta\sum_{\sigma=0}^s\overline{\Delta_{n,h,\sigma}}}\big]-\Exd\big[\ee^{\ii\zeta\sum_{\sigma=0}^{s-1}\overline{\Delta_{n,h,\sigma}}}\big]
  \Exd\big[\ee^{\ii\zeta\overline{\Delta_{n,h,s}}}\big]\Big|\\
  \nonumber
  &\qquad\le \frac{\zeta^2}{nw(h)}\sum_{a=1}^{l_{s-1}+m_n}\sum_{b=l_s+1}^{l_s+m_n}\sum_{i=1}^n\sum_{j=1}^n\big|\gamma(|i-j|)\big|
  \int_0^1\dd z \int_{\Omega}\probd[\dd\omega]\int_{\Omega}\probd[\dd\omega']\\
  &\qquad\quad\times\big|\cov_{n,h,\omega}[X_a,X_i]\big|\big|\cov_{n,h,z\omega+\sqrt{1-z^2}\omega'}[X_b,X_i]\big|\,.
\label{eq:CLT_centering_131}
\end{align}
Thanks to well-known properties of Gaussian random variables, for any
$z\in[0,1]$ the sequence $z\omega+\sqrt{1-z^2}\omega'$ is an outcome
of an environment with law $\probd$. Then, an application of the
Cauchy-Schwarz inequality to \eqref{eq:CLT_centering_131} allows us to
obtain
\begin{align*}
  &\Big|\Exd\big[\ee^{\ii\zeta\sum_{\sigma=0}^s\overline{\Delta_{n,h,\sigma}}}\big]-\Exd\big[\ee^{\ii\zeta\sum_{\sigma=0}^{s-1}\overline{\Delta_{n,h,\sigma}}}\big]
  \Exd\big[\ee^{\ii\zeta\overline{\Delta_{n,h,s}}}\big]\Big|\\
  &\qquad\le \frac{\zeta^2}{nw(h)}\sum_{a=1}^{l_{s-1}+m_n}\sum_{b=l_s+1}^{l_s+m_n}\sum_{i=1}^n\sum_{j=1}^n\big|\gamma(|i-j|)\big|
  \sqrt{\Exd\big[\cov_{n,h,\cdot}[X_a,X_i]^2\big]\Exd\big[\cov_{n,h,\cdot}[X_b,X_j]^2\big]}\\
  &\qquad\le \frac{C(1-\ee^{-\lambda})^3\zeta^2}{nw(h)}\sum_{a=1}^{l_{s-1}+m_n}\sum_{b=l_s+1}^{l_s+m_n}\sum_{i\in\N}\sum_{j\in\N}\big|\gamma(|i-j|)\big|\ee^{-\lambda|a-i|-\lambda|b-j|}\,,
\end{align*}
where the second bound is due to \eqref{eq:CLT_centering_1}. At this
point, invoking Lemma \ref{lem:sums} and noting that
$l_{s-1}+m_n=l_s-d_n$, we deduce that
\begin{align*}
  \Big|\Exd\big[\ee^{\ii\zeta\sum_{\sigma=0}^s\overline{\Delta_{n,h,\sigma}}}\big]-\Exd\big[\ee^{\ii\zeta\sum_{\sigma=0}^{s-1}\overline{\Delta_{n,h,\sigma}}}\big]
  \Exd\big[\ee^{\ii\zeta\overline{\Delta_{n,h,s}}}\big]\Big|&\le \frac{16C\bar\gamma\zeta^2}{nw(h)}+\frac{4C\zeta^2}{nw(h)}\sum_{i=1}^{l_s-d_n}\sum_{j=l_s+1}^{l_s+m_n}\big|\gamma(j-i)\big|\\
  &\le \frac{16C\bar\gamma\zeta^2}{nw(h)}+\frac{4C\zeta^2m_n}{nw(h)}\sum_{a=d_n+1}^\infty|\gamma(a)|\,.
\end{align*}
Carrying out the sum over $s\in\{1,\ldots,r_n-1\}$ and recalling that
$r_nm_n\le n$ we finally find for all $n\in\N$, $h\in H$, and
$\zeta\in\Rl$
\begin{align*}
  &\sum_{s=1}^{r_n-1}\Big|\Exd\big[\ee^{\ii\zeta\sum_{k=0}^s\overline{\Delta_{n,h,\sigma}}}\big]-\Exd\big[\ee^{\ii\zeta\sum_{k=0}^{s-1}\overline{\Delta_{n,h,\sigma}}}\big]
  \Exd\big[\ee^{\ii\zeta\overline{\Delta_{n,h,s}}}\big]\Big|\\
&\qquad\qquad\qquad\qquad\qquad\qquad\qquad\le  \frac{16C\bar\gamma\zeta^2}{w(h)}\frac{r_n}{n}+\frac{4C\zeta^2}{w(h)}\sum_{a=d_n+1}^\infty|\gamma(a)|\,,
\end{align*}
which is exactly \eqref{eq:CLT_centering_13}.
\end{proof}


\bibliographystyle{abbrv}
\bibliography{biblio.bib}

\end{document}